\documentclass{article}
\usepackage[utf8]{inputenc}
\usepackage{amsmath}
\usepackage{amssymb}
\usepackage{amsthm}
\usepackage{mathtools}
\usepackage{color}
\usepackage{lineno}

\usepackage{authblk}
\usepackage{blindtext}

\newcommand{\C}{{\mathbb C}}
\newcommand{\R}{{\mathbb R}}

\newcommand{\Z}{{\mathbb Z}}

\newcommand{\g}{{\mathfrak g}}

\newcommand{\p}{{\mathfrak p}}
\newcommand{\fk}{{\mathfrak k}}

\newcommand{\lv}{{\mathfrak l}}

\newcommand{\Ind}{\operatorname{Ind}}
\newcommand{\Sym}{\operatorname{Sym}}

\newcommand{\Hom}{\operatorname{Hom}}

\newcommand{\tl}{\tau_{\Lambda}}

\newcommand{\sgn}{\operatorname{sgn}}

\newcommand{\cD}{{\cal D}}

\newcommand{\cH}{{\cal H}}

\newcommand{\diag}{\operatorname{diag}}
\numberwithin{equation}{section}
\theoremstyle{plain}
 \newtheorem{thm}{Theorem}[section]
 \newtheorem{prop}[thm]{Proposition}
 \newtheorem{lem}[thm]{Lemma}
 \newtheorem{cor}[thm]{Corollary}

\theoremstyle{definition}
 
 \newtheorem{rem}[thm]{Remark}

\newcommand{\bbA}{\mathbb{A}}
\newcommand{\bbC}{\mathbb{C}}
\newcommand{\bbR}{\mathbb{R}}
\newcommand{\bbQ}{\mathbb{Q}}
\newcommand{\bbZ}{\mathbb{Z}}

\newcommand{\calA}{\mathcal{A}}

\newcommand{\calD}{\mathcal{D}}

\newcommand{\calL}{\mathcal{L}}
\newcommand{\calN}{\mathcal{N}}

\newcommand{\calU}{\mathcal{U}}

\newcommand{\calZ}{\mathcal{Z}}

\newcommand{\fraka}{\mathfrak{a}}

\newcommand{\frakg}{\mathfrak{g}}
\newcommand{\frakH}{\mathfrak{H}}

\newcommand{\frakk}{\mathfrak{k}}
\newcommand{\frakl}{\mathfrak{l}}
\newcommand{\frakL}{\mathfrak{L}}

\newcommand{\frakn}{\mathfrak{n}}

\newcommand{\frakp}{\mathfrak{p}}
\newcommand{\fraks}{\mathfrak{s}}

\newcommand{\frakt}{\mathfrak{t}}

\newcommand{\GL}{\mathrm{GL}}
\newcommand{\Sp}{\mathrm{Sp}}

\newcommand{\SL}{\mathrm{SL}}
\newcommand{\SO}{\mathrm{SO}}

\newcommand{\fini}{\mathrm{fin}}

\newcommand{\Lie}{\mathrm{Lie}}

\newcommand{\bs}{\backslash}

\newcommand{\Wh}{\mathrm{Wh}}
\newcommand{\cusp}{\mathrm{cusp}}

\newcommand{\vep}{\varepsilon}
\newcommand{\isom}{\cong}

\title{Cuspidal components of Siegel modular forms for large discrete series representations of $\Sp_4(\mathbb{R})$}
\author[1]{Shuji Horinaga \thanks{syuuji.horinaga@ntt.com, shorinaga@gmail.com}}
\author[2]{Hiro-aki Narita \thanks{hnarita@waseda.jp}}
\date{\today}
\affil[1]{NTT Institute for Fundamental Mathematics, NTT Communication Science Laboratories, NTT Corporation, Japan}
\affil[2]{Department of Mathematics
Faculty of Science and Engineering
Waseda University
3-4-1 Okubo, Shinjuku-ku, Tokyo 169-8555
JAPAN}

\begin{document}

\maketitle

\begin{abstract}
    In this paper, we consider automorphic forms on $\Sp_4(\bbA_\bbQ)$ which generate large discrete series representations of $\Sp_4(\bbR)$ as $(\mathfrak{sp}_4(\bbR),K_\infty)$-modules.
    We determine the cuspidal components and the structure of the space of such automorphic forms.
\end{abstract}


\section{Introduction}
    
    A Siegel modular form is a generalization of modular forms of one variable and is a major topic of interest in number theory.
    Its role in number theory, including L-functions, quadratic forms, and many problems, has been very significant.
    Harish-Chandra, Langlands, and others have found that it is useful to understand it from the viewpoint of representation theory.
    In this context, Siegel modular forms correspond to the highest weight $K$-types of holomorphic discrete series representations on symplectic groups.
    However, the concepts of the modular forms corresponding to discrete series representations other than holomorphic ones is not well understood.
    In this paper, we develop a structural theory of the modular forms for large discrete series representations of $\Sp_4(\bbR)$, the symplectic group of degree two.
    
    We first recall the theory of Siegel modular forms of weight $k$ with respect to the full modular group $\Gamma_n=\Sp_{2n}(\bbZ)$.
    Let $M_k(\Gamma_n)$ be the space of Siegel modular forms of weight $k$ of degree $n$ with respect to $\Gamma_n$ and $S_k(\Gamma_n)$ be the subspace of cusp forms in $M_k(\Gamma_n)$.
    Suppose that $k$ is even.
    Take $0 \leq \ell < n$ and $f \in S_k(\Gamma_\ell)$.
    Put
    \[
    P = \left\{
    \left(
    \begin{array}{cc|cc}
    A&*&*&*\\
    0_{n - \ell,\ell}&a&*&b\\
    \hline
    0_{n - \ell}&0_{\ell,n-\ell}&{^t A}^{-1}&0_{\ell,n-\ell}\\
    0_{n-\ell,\ell}&c&*&d
    \end{array}
    \right)
    \in \Sp_{2n}(\bbZ) \middle|\, A \in \GL_{n-\ell}(\bbZ), \begin{pmatrix}a & b \\ c & d \end{pmatrix} \in \Sp_{2\ell}(\bbZ)\right\}.
    \]
    We now define the Klingen Eisenstein series $E_k^{(n)}(z,f)$ of weight $k$ attached to $f$ by
    \[
    E_k^{(n)}(z,f)
    =
    \sum_{\gamma \in P \cap \Gamma_n \bs \Gamma_n} j(\gamma,z)^{-k}f((\gamma(z))^{*}),
    \]
    where $z^* = z_3$ for $z = \left(\begin{smallmatrix}z_1&z_2\\ {^t z_2} & z_3\end{smallmatrix}\right)$, and $j$ is the factor of automorphy.
    If $k > n+\ell+1$, the sum converges absolutely and uniformly on any compact sets in the Siegel upper half space $\frakH_n$.
    Let $E_{k}^{(\ell)}(\Gamma_n)$ be the space spanned by Klingen Eisenstein series $E_k^{(n)}(z,f)$ of weight $k$ attached to $f$ for all $f \in S_k(\Gamma_\ell)$.
    The following structure theorem for Siegel modular forms is well-known (cf.~\cite[Chap.~II, Proposition 6]{90_Klingen}), 
    where we set $E_k^{(n)}(\Gamma_n)=S_k(\Gamma_n)$.
    
    \begin{thm}
    Let $k$ be a weight.
    If $k$ is even with $k \geq 2n+2$, we have
    \[
    M_k(\Gamma_n) = \bigoplus_{\ell=0}^n E_k^{(\ell)}(\Gamma_n).
    \]
    \end{thm}
    
    This theorem can be generalized to nearly holomorphic automorphic forms on $\Sp_{2n}(\bbA_F)$ for a totally real field $F$.
    For details, see \cite{Horinaga_2, Horinaga_3}.
    
    We now consider the automorphic forms $\varphi$ on $\Sp_4(\bbA_\bbQ)$ such that $\varphi$ generates a large discrete series representation of $\Sp_{4}(\bbR)$ at the archimedean place.
    The discrete series representations of $\Sp_{4}(\bbR)$ is parametrized by $\{(\lambda_1,\lambda_2) \in \bbZ^{2}\mid \lambda_1 \geq \lambda_2\} \setminus (\{\lambda_1 \cdot \lambda_2=0\}\cup \{\lambda_1 = \pm \lambda_2\})$.
    For simplicity, in this section, we only treat the case where $\lambda \in \{(\lambda_1,\lambda_2)\in \bbZ_{>0}\times \bbZ_{<0} \mid \lambda_1 > -\lambda_2\}$.
    The corresponding discrete series representation is defined to be of type II and large.
    We denote the representation corresponding to $\lambda$ by $\calD_\lambda$.
    Let $\calL_\lambda$ be the space of automorphic forms $\varphi$ on $\Sp_{4}(\bbA_\bbQ)$ such that $\varphi$ generates $\calD_\lambda$.
    By the general theory of automorphic forms (cf.~\cite{MW}), the space $\calL_\lambda$ decomposes as the direct sum
    \[
    \calL_\lambda = \bigoplus_{(M,\pi)} \calL_{\lambda,(M,\pi)},
    \]
    where $(M,\pi)$ runs over all equivalence classes of cuspidal data and $\calL_{\lambda, (M,\pi)}$ is the subspace of $\calL_\lambda$ spanned by automorphic forms with the cuspidal support $(M,\pi)$.
    Here, a cuspidal datum $(M,\pi)$ consists of a Levi subgroup $M$ and an irreducible cuspidal automorphic representation $\pi$ of $M(\bbA_\bbQ)$.
    For a parabolic subgroup $P$ of $\Sp_4(\bbA_\bbQ)$ and an automorphic form on $\Sp_4(\bbA_\bbQ)$, let $\varphi_P$ be the constant term of $\varphi$ along $P$.
    Then, by theorems in \S\ref{section_Whitt}, the constant term $\varphi_P$ lies in the induced representation.
    Conversely, for an element $f$ of induced representations, we can define the Eisenstein series as
    \[
    E(g,f) = \sum_{\gamma \in P(\bbQ) \bs \Sp_{4}(\bbQ)} f(\gamma g).
    \]
    We now state the main theorem for the case where $M=\GL_2$ and $\lambda \in \Xi_{II}$.
    By $\calD_k$, we mean the discrete series representation of $\GL_2(\bbR)$ of weight $k \in \bbZ_{>1}$ with trivial central character.
    For the other cases, see \S \ref{section_main}.
    
    \begin{thm}
    Let $P=M_PN_P$ be the Siegel parabolic subgroup of $\Sp_{4}$~(cf.~Section \ref{Symp-gp}) and $\lambda = (\lambda_1,\lambda_2) \in \Xi_{II}$.
    Take an irreducible cuspidal automorphic representation $\pi=\otimes_v\pi_v$ of $\GL_2(\bbA_\bbQ)$ such that $\pi$ is invariant under the split component of $M_P(\bbA_\bbQ)=\GL_2(\bbA_\bbQ)$.
    \begin{enumerate}
        \item If $\calL_{\lambda,(M_P,\pi)}$ is non-zero, the archimedean component $\pi_\infty$ is a discrete series representation of $\GL_2(\bbA_\bbQ)$ of weight $\lambda_1+\lambda_2+1$ or $\lambda_1-\lambda_2+1$.
        \item If $\pi_\infty = \calD_{\lambda_1+\lambda_2+1}$ and $\lambda_1-\lambda_2>3$, the constant term along $P$ induces the isomorphism
	    \[
	    \calL_{\lambda,(M,\pi)} \isom \left(\bigotimes_{v < \infty}\Ind_{P_S(\bbQ_b)}^{G(\bbQ_v)}\left(|\cdot|^{(\lambda_1-\lambda_2)/2}\otimes\pi_v\right)\right) \otimes \calD_\lambda.
	    \]
	    \item If $\pi_\infty = \calD_{\lambda_1-\lambda_2+1}$ and $\lambda_1+\lambda_2>3$, the constant term along $P$ induces the isomorphism
	    \[
	    \calL_{\lambda,(M,\pi)} \isom \left(\bigotimes_{v < \infty}\Ind_{P_S(\bbQ_v)}^{G(\bbQ_v)}\left(|\cdot|^{(\lambda_1+\lambda_2)/2}\otimes\pi_v\right)\right) \otimes \calD_\lambda.
	    \]
	\end{enumerate}
    \end{thm}
    
    In \cite{Horinaga_1, Horinaga_2, Horinaga_3}, similar structure theorems for nearly holomorphic automorphic forms on $\Sp_{2n}$ are proved.
    We clarify the differences between the case for large discrete series representations and for nearly holomorphic modular forms in Remark \ref{diff_NH}.
    
    As far as the authors know, there has been no result of an explicit description of cuspidal components for non-holomorphic automorphic forms.
    An additional novelty of our research is that we examine the Whittaker functions attached to degenerate characters of the maximal unipotent subgroup of $\Sp_4(\R)$~(cf.~Section \ref{Deg-Whitt}).
    Usually the Whittaker functions for a quasi-split reductive group $G$ over a local field are defined for non-degenerate characters of a maximal unipotent subgroup of $G$.
    The notion of Whittaker functions is reformulated as the Whittaker model for an admissible representation of a quasi-split  group over a local field.
    For the case of a quasi-split real reductive group the multiplicity free property is known for irreducible admissible representations~(cf.~Wallach \cite[Theorem 8.8]{1983_Wallach}). 
    However, according to relevant prior results~(cf.~\cite{2001_Hirano, 2009_Muic}) as well as our results, the multiplicity free property seems to collapse frequently for Whittaker functions of degenerate cases.
    
\subsection*{Acknowledgement}
    This work was supported by the Research Institute for Mathematical Sciences, an International Joint Usage/Research Center located in Kyoto University.
    The first author is supported by AIP Challenge of JST CREST JPMJCR14D6, JST CREST JPMJCR14D6 and CREST JPMJCR2113, Japan. 
    The second named author was partially supported by Grand-in-Aid for Scientific Research (C) 19K03431, Japan Society for the Promotion of Science.  
    We are very grateful to Professor Taku Ishii for his generosity to use a part of the results in the paper about the Whittaker functions on $\Sp_4(\mathbb{R})$ attached to degenerate characters, being prepared jointly with the second named author. 
    
\section{Notation}\label{BN}
   
    For what follows, we introduce a basic notation. In addition to fixing the convention of notation for groups over local fields such as real groups we need notations for the global theory, e.g. ad\'ele groups etc. 
    
\subsection{Basic notation}\label{Basic-Notation}
    
    For the field of rational numbers $\bbQ$, we denote by $\bbA$ the ring of ad\'eles of $\bbQ$.
    Let $\bbA_{\fini}$ be the finite part of $\bbA$.
    For simplicity, we say that one dimensional representation is a character.
    We do not assume that a character is unitary.
    By a Hecke character we mean a character of $\bbQ^\times \bbR_+^\times \bs \bbA^\times$.
    We define a non-trivial additive character $\psi=\otimes_v\psi_v$ of $\bbQ\bs\bbA$ by $\psi_\infty (x)= \exp(2\pi \sqrt{-1}\,x)$ and $\psi_p(x)=\exp(-2\pi\sqrt{-1}y)$ for $x \in \bbQ_v$.
    Here, $p$ is a rational prime and $y$ is an element of $\cup_{n=1}^\infty p^{-n}\bbZ$ such that $x-y \in \bbZ_p$.
    For an algebraic group $G$ over $\bbQ$ and a place $v$ of $\bbQ$, set $G_v$ to be the $\bbQ_v$-valued points of $G$.
    For the sake of simplicity, when the place $v$ is obvious in context, we write $G=G_v$.
    Specifically, in Sections \ref{section_rep_Lie_group}, \ref{section_emb_ind_rep} and \ref{section_Whitt}, we only treat groups at the archimedean place $\infty$.
    For an automorphic representation of $G(\bbA)$, we mean a subspace of automorphic forms on $G(\bbA)$ stable under the right translation by  $G(\bbA_\fini)$ and admitting a $(\Lie(G(\bbR)),K)$-module structure at the archimedean place. 
    Here, $\Lie(G(\bbR))$ denotes the Lie algebra of $G(\bbR)$ and $K$ is a maximal compact subgroup of $G(\bbR)$.
    Let $G$ be a reductive algebraic group over a local field $F$.
    For a parabolic subgroup $P$ of $G$ over $F$, we denote by $\delta_P$ the modulus character of $P$.
    Take a representation $\pi$ of a Levi subgroup $M(F)$ of $P(F)=M(F)N(F)$.
    Let $\Ind_{P(F)}^{G(F)}(\pi)$ be the normalized induced representations, i.e.,
    \[
    \Ind_{P(F)}^{G(F)}(\pi) = \{f \colon G(F) \longrightarrow \pi \mid \text{$f(mng)=\delta_P(m)\pi(m)f(g)$ for any $m \in M(F)$ and $n \in N(F)$}\}.
    \]
    
\subsection{Symplectic groups}\label{Symp-gp}
    
   By $R$ we denote a ring. Let $G$ be the symplectic group defined by
    \[
    G(R)=\{g\in \GL_4(R)\mid {}^tgJ_4g=J_4\},
    \]
    where $J_4=
    \left(
    \begin{smallmatrix}
    0_2 & \mathrm{1}_2\\
    -\mathrm{1}_2 & 0_2
    \end{smallmatrix}
    \right)$. 
    This has two maximal parabolic subgroups called the Jacobi~(or Klingen) parabolic subgroup $P_J$ and the Siegel parabolic subgroup $P_S$ up to conjugation. 
    
    We first provide a review of the group $P_J$ and its subgroups.
    The group $P_J$ has the Levi decomposition $N_J\rtimes L_J$. Here $N_J$ is the nilpotent algebraic group defined by
    \[
    N_J(R)=
    \left\{\left. n(u_0,u_1,u_2):=
    \left(
    \begin{array}{cc|cc}
    1 & 0 & u_1 & u_2\\
    0 & 1 & u_2 & 0\\
    \hline
    0 & 0 & 1 & 0\\
    0 & 0 & 0 & 1
    \end{array}
    \right)
    \left(
    \begin{array}{cc|cc}
    1 & u_0 & 0 & 0\\
    0 & 1 & 0 & 0\\
    \hline
    0 & 0 & 1 & 0\\
    0 & 0 & -u_0 & 1
    \end{array}
    \right)~\right|~u_0,~u_1,~u_2\in R\right\}
    \]
    and the Levi part $L_J$ is the subgroup of $G$ given by
    \begin{align}\label{def_L_J}
    L_J(R)=
    \left\{\left.
    \left(
    \begin{array}{cc|cc}
    \alpha & & & \\
    & a & & b\\
    \hline
    & & \alpha^{-1} & \\
    & c & & d
    \end{array}
    \right)~\right|~\alpha\in R^{\times},~
    \begin{pmatrix}
    a & b\\
    c & d
    \end{pmatrix}\in \SL_2(R)\right\}.
    \end{align}
    The unipotent radical $N_J$ of $P_J$ is nothing but the Heisenberg group with the center  
    \[
    Z_J:=\{n(0,u_1,0)\mid u_1\in R\}.
    \]
    
    We introduce the Jacobi group $G_J$ defined by the semi-direct product 
    \[
    G_J(R) = N_J(R)\rtimes \SL_2(R),
    \]
    where $\SL_2(R)$ is viewed as a subgroup of $G_J(R)$~(or $L_J(R)$) by putting $\alpha=1$ in (\ref{def_L_J}).
    The groups $G_J$ and $N_J$ have  $Z_J$ as the center in common. 

    We next review another maximal parabolic subgroup $P_S$, which has the Levi decomposition $N_S\rtimes L_S$.
    Here $N_S$ is the nilpotent algebraic group defined by
    \[
    N_S(R)=
    \left\{ n_S(u_0,u_1,u_2):=
    \left(
    \begin{array}{cc|cc}
    1 & 0 & u_0 & u_1\\
    0 & 1 & u_1 & u_2\\
    \hline
    0 & 0 & 1 & 0\\
    0 & 0 & 0 & 1
    \end{array}
    \right)
    \,\middle|\, u_0, u_1, u_2\in R\right\}
    \]
    and the Levi part $L_J$ is the subgroup of $G$ given by
    \[
    L_J(R)=
    \left\{
    \begin{pmatrix}
    A&\\
    &{^t}A^{-1}
    \end{pmatrix}
    \,\middle|\,
    A \in \GL_2(R)
    \right\}.
    \]
    The unipotent radical is abelian.
    
    We fix a minimal parabolic subgroup $P_0$ as follows.
    The group $P_0$ has the unipotent radical $N_0$ and a Levi subgroup $L_0$, where $L_0$ is the group of diagonal matrices in $G$ and $N_0$ is defined by 
    \[
    N_0(R)=
    \left\{\left.n(u_0,u_1,u_2,u_3):=
    \left(
    \begin{array}{cc|cc}
    1 & 0 & u_1 & u_2\\
    0 & 1 & u_2 & u_3\\
    \hline
    0 & 0 & 1 & 0 \\
    0 & 0 & 0 & 1
    \end{array}
    \right)
    \left(
    \begin{array}{cc|cc}
    1 & u_0 & 0 & 0\\
    0 & 1 & 0 & 0\\
    \hline
    0 & 0 & 1 & 0\\
    0 & 0 & -u_0 & 1
    \end{array}
    \right)~\right|~\text{$n_i\in R$ for $0\le i\le 3$}\right\}.
    \]
    A parabolic subgroup $P$ is called standard if $P$ contains $P_0$.

\subsection{Real symplectic groups}\label{Real-spgp}
    
    In this subsection, we focus on the real case, i.e., $R=\bbR$, and put $H=H(\bbR)$ for algebraic subgroups $H$ of $G$, in particular $G:=G(\bbR)$.
    We first review the Langlands decomposition of parabolic subgroups of $G$.
    The parabolic subgroup $P_J$ has the Langlands decomposition $P_J=N_JA_J^\infty M_J$ with 
    \[
    A_J^\infty:=\left\{\left.
    \left(
    \begin{array}{cc|cc}
    a & 0 & 0 & 0\\
    0 & 1 & 0 & 0\\
    \hline
    0 & 0 & a^{-1} & 0\\
    0 & 0 & 0 & 1
    \end{array}
    \right)
    ~\right|~a\in\R^{\times}_+\right\},~
    M_J:=\left\{\left.
    \left(
    \begin{array}{cc|cc}
    \vep & 0 & 0 & 0\\
    0 & a & 0 & b\\
    \hline
    0 & 0 & \vep & 0\\
    0 & c & 0 & d
    \end{array}
    \right)
    ~\right|~
    \begin{array}{c}
    \begin{pmatrix}
    a & b\\
    c & d
    \end{pmatrix}\in \SL_2(\R)\\
    \vep\in\{\pm 1\}
    \end{array}\right\}.
    \]
    
    The Langlands decomposition of $P_S$ is given by $P_S=N_SA_S^\infty M_S$ with 
    \[
    A_S^\infty:=\{\diag(a,a,a^{-1},a^{-1})\mid a\in\R_{>0}\},~
    M_S:=\left\{\left.
    \begin{pmatrix}
    A & 0_2\\
    0_2 & {}^tA^{-1}
    \end{pmatrix}~\right|~A\in \GL(2,\R),~\det(A)=\pm1\right\}.
    \]
    We will use the notation $\SL_2^{\pm}(\R):=\{A\in \GL_2(\R)\mid \det(A)=\pm1\}$. We obviously have $M_S\simeq \SL_2^{\pm}(\R)$. 
    
    We also review the Langlands decomposition $P_0=N_0A_0^\infty M_0$ of $P_0$, where
    \[
    A_0^\infty:=\{a_0=\diag(a_1,a_2,a_1^{-1},a_2^{-1})\mid a_1,~a_2\in\R_{>0}\},~M_0:=\{\diag(\vep_1,\vep_2,\vep_1,\vep_2)\mid \vep_1,~\vep_2\in\{\pm 1\}\}.
    \]
    We now note that $N_S=\{n(u_0,u_1,u_2,u_3)\in N_0\mid u_0=0\}$. 
    The group $N_0$ admits the semi-direct product decomposition $N_0=N_S\rtimes N_L$ with the subgroup $N_L$ defined by $N_L:=\{n(u_0,0,0,0)\mid u_0\in \bbR\}$.
    The split component $A_H$ of an algebraic group $H$ is defined by the split torus of the center of $H$.
    For $* \in \{0,S,J\}$, the group $A_*^\infty$ is the identity component of the split component of $P_*$ in the real topology.
    
    Let us introduce the Cartan involution $\theta$ of $G$ defined by $\theta(g):={}^tg^{-1}$ for $g\in G$.
    Then the group 
    \[
    K:=\{g\in G\mid\theta(g)=g\}=\left\{\left.
    \begin{pmatrix}
    A & B\\
    -B & A
    \end{pmatrix}\in G~\right|~A,B\in \mathrm{M}_2(\R)\right\}
    \]
    is a maximal compact subgroup of $G$.
    This is isomorphic to the unitary group $\mathrm{U}(2)$ of degree two  by the map
    \[
    K\ni
    \begin{pmatrix}
    A & B\\
    -B & A
    \end{pmatrix}\mapsto A+\sqrt{-1}\,B\in \mathrm{U}(2).
    \]
    Note that $G$ has an Iwasawa decomposition $G=N_0A_0^\infty K$ with the notation above.

    Following the standard manner of the notation, we denote the Lie algebras of real groups by the corresponding German letter~(Fraktur).
    For a real Lie algebra $\lv$, we denote its complexification by $\lv_{\C}$. 
    The Lie algebra $\mathfrak{g}$ of $G(\bbR)$ is given by $\{X\in \mathrm{M}_4\mid {}^tXJ_4+J_4X=0_4\}$.
    The Cartan involution of $\g$, also denoted by $\theta$, is defined by $\theta(X)=-{}^tX$ for $X\in\g$. Then $\g$ has the eigen-space decomposition $\g=\fk+\p$ with 
    \begin{align*}
    \fk&=\{X\in\g\in\mid \theta(X)=X\}=\left\{\left.
    \begin{pmatrix}
    A & B\\
    -B & A
    \end{pmatrix}~\right|~A,~B\in \mathrm{M}_2(\R),~{}^tA=-A,~{}^tB=B\right\},\\
    \p&=\{X\in\g\in\mid \theta(X)=-X\}=\left\{\left.
    \begin{pmatrix}
    A & B\\
    B & -A
    \end{pmatrix}~\right|~A,~B\in \mathrm{M}_2(\R),~{}^tA=A,~{}^tB=B\right\}.
    \end{align*}
    The former is nothing but the Lie algebra of $K$. Here we give a basis of $\frakk$ as follows:
    \begin{align*}
        K_{11}&:=
        \left(
        \begin{array}{cc|cc}
         & & -\sqrt{-1} &\\
         & & & 0\\
         \hline
         \sqrt{-1} & & &\\
         & 0 & &
        \end{array}
        \right),\quad
        K_{22}:=
        \left(
        \begin{array}{cc|cc}
        & & 0 &\\
         & & & -\sqrt{-1}\\
         \hline
         0 & & &\\
         & \sqrt{-1} & &
        \end{array}
        \right),
        \\
        K_{12}&:=\frac{1}{2}
        \left(
        \begin{array}{cc|cc}
        & 1 & & -\sqrt{-1}\\
        -1 & & -\sqrt{-1} &\\
        \hline
        & \sqrt{-1} & & 1\\
        \sqrt{-1} & & -1 &
        \end{array}
        \right),\quad 
        K_{21}:=\frac{1}{2}
        \left(
        \begin{array}{cc|cc}
        & -1 & & -\sqrt{-1}\\
        1 & & -\sqrt{-1} &\\
        \hline
        & \sqrt{-1} & & -1\\
        \sqrt{-1} & & 1 &
        \end{array}
        \right).
    \end{align*}

    We consider the root space decomposition of $\g_{\C}$ with respect to the complexification $\mathfrak{t}_{\C}$ of the compact Cartan subalgebra $\mathfrak{t}=\R T_1\oplus\R T_2$~(in $\fk$) with $T_1:=\sqrt{-1}\,K_{11}$ and $T_2:=\sqrt{-1}\,K_{22}$. 
    The dual space ${\mathfrak t}_{\C}^*$ of ${\mathfrak t}_{\C}$ has a basis $\{\beta_1,\beta_2\}$ given by
    \[
    \beta_i(T_j)=\sqrt{-1}\,\delta_{ij}.
    \]
    We denote $\beta\in{\mathfrak t}_{\C}^*$ by $(a,b)$ if $\beta=a\beta_1+b\beta_2$. 
    With this notation, the set of roots for the root space decomposition $(\g_{\C},{\mathfrak t}_{\C})$ is given by $\Delta:=\{\pm(2,0),~\pm(0,2),~\pm(1,1),~\pm(1,-1)\}$.
    This set has the standard choice of positive roots given by $\Delta^+:=\{(2,0),~(0,2),~(1,1),~(1,-1)\}$. 
    The set of roots $\{\pm(1,-1)\}$ forms the set of compact roots, whose root vectors are in $\fk_{\C}$. 
    Each root in $\{\pm(2,0),~\pm(0,2),~\pm(1,1)\}$ is called a non-compact root, whose root vector is in the complexification $\p_{\C}$ of $\p$.
    There is a standard choice of the root vectors:
    \[
    X_{\pm(2,0)}:=p_{\pm}\left(
    \begin{pmatrix}
    1 & 0\\
    0 & 0
    \end{pmatrix}\right),~
    X_{\pm(1,1)}:=
    p_{\pm}
    \left(
    \begin{pmatrix}
    0 & 1\\
    1 & 0
    \end{pmatrix}
    \right),~
    X_{\pm(0,2)}:=p_{\pm}
    \left(
    \begin{pmatrix}
    0 & 0\\
    0 & 1
    \end{pmatrix}
    \right),
    \]
    where we put
    \[p_{\pm}(X):=
    \begin{pmatrix}
    X & \pm\sqrt{-1}\,X\\
    \pm\sqrt{-1}\,X & -X
    \end{pmatrix}\in \mathrm{M}_4(\C)
    \]
    for real symmetric matrices $X$ of degree two.

    We also need the restricted root space decomposition with respect to the abelian Lie algebra $\fraka:=\R H_1\oplus\R H_2$ in $\frakp$ with \[
    H_1:=\diag(1,0,-1,0),~H_2:=\diag(0,1,0,-1).
    \]
    The restricted root system is given by $\Delta(\frakg,\fraka):=\{\pm2e_1,~\pm2e_2,~\pm e_1\pm e_2\}$ with the linear forms $e_1,~e_2$ of $\fraka$ defined by $e_i(H_j)=\delta_{ij}$ for $1\le i,j\le 2$. This is of the same type as the complex root system above.
    Let $\frakn:=\R E_{e_1-e_2}\oplus\R E_{e_1+e_2}\oplus\R E_{2e_1}\oplus\R E_{2e_2}$ be the subspace of $\frakg$ spanned by the root vectors 
    \[
    E_{e_1-e_2}:=E_{12}-E_{43},~E_{e_1+e_2}:=E_{14}+E_{23},~E_{2e_1}:=E_{13},~E_{2e_2}:=E_{24}
    \]
    for positive roots of the restricted root system, where $E_{ij}$ denotes the matrix unit indexed by $1\le i,j\le 4$.
    We have an Iwasawa decomposition
    \[
    \frakg=\frakn\oplus\fraka\oplus\frakk.
    \]
    For later use, we prepare the Iwasawa decomposition of the root vectors in $\frakp$ as follows:
    
    \begin{lem}\label{Iwasawa-decomp}
    For roots in $\Delta(\frakg_\bbC, \frakt_\bbC)$, root vectors can be described as follows:
    \begin{align*}
    &X_{\pm(2,0)}=\pm2\sqrt{-1}\,E_{2e_1}+H_1\pm K_{11},\quad X_{\pm(0,2)}=\pm2\sqrt{-1}\,E_{2e_2}+H_2\pm K_{22},\\
    &X_{(1,1)}=2(E_{e_1-e_2}+\sqrt{-1}\,E_{e_1+e_2})+2K_{21},\quad 
    X_{-(1,1)}=2(E_{e_1-e_2}-\sqrt{-1}\,E_{e_1+e_2})-2K_{12}.
    \end{align*}
    \end{lem}


\section{Representations of real groups}\label{section_rep_Lie_group}
    
    This section presents fundamental facts on representations of real groups, which we use to describe the archimedean aspect of this paper.
    As mentioned in Section \ref{Basic-Notation}, we keep the notation of real groups until Section \ref{section_Whitt}.
    
\subsection{Discrete series representations of $\SL_2(\R)$ and $\SL_2^{\pm}(\R)$}\label{DS-SL(2)}
    
    For $n\in\Z_{>1}$, we let $\cD_n^+$~(resp.~$\cD_n^-$) be a discrete series representation of $\SL_2(\R)$ with lowest weight $n$~(resp.~highest weight $-n$).
    The discrete series representation $\cD_n^{\pm}$ has a representation space with a basis $\{w_\ell\mid \ell\in \pm n\pm\Z_{\ge 0}\}$ satisfying
    \[
    \cD^{\pm}_n(U)w_\ell=\ell w_\ell,\quad\cD^{\pm}_n(V^{\pm})w_\ell=(\pm n\pm \ell)w_{\ell\pm 2},
    \]
    where
    \[
    U:=
    \begin{pmatrix}
    0 & -\sqrt{-1}\\
    \sqrt{-1} & 0
    \end{pmatrix},\quad
    V^{\pm}:=\frac{1}{2}
    \begin{pmatrix}
    1 & \pm\sqrt{-1}\\
    \pm\sqrt{-1} & -1
    \end{pmatrix}.
    \]
    Recall that we have introduced  $\SL^{\pm}_2(\R):=\{g\in \GL_2(\R)\mid\det(g)\in\{\pm 1\}\}$.
    Discrete series representations of $\SL_2^{\pm}(\R)$ are of the form
    \[
    {\rm Ind}_{\SL_2(\R)}^{\SL_2^{\pm}(\R)}\cD_n^+\simeq {\rm Ind}_{\SL_2(\R)}^{\SL_2^{\pm}(\R)}\cD_n^-,
    \]
    which is isomorphic to $\cD_n^+\oplus\cD_n^-$ as unitary representations of $\SL_2(\R)$.
    We denote this discrete series representation by $\cD_n$.
    
\subsection{Representations of the maximal compact subgroup $K$}\label{RepMaxCpt}
    
    Irreducible finite dimensional representations of a compact linear connected real reductive group are parametrized by dominant weights.
    For the maximal compact subgroup $K$ of $G$, the set of equivalence classes of irreducible finite dimensional representations of $K$ are in bijection with the set of the dominant weights $\{(\Lambda_1,\Lambda_2)\in\Z^2\mid\Lambda_1\ge\Lambda_2\}$, where $(\Lambda_1,\Lambda_2)$ denotes the root $\Lambda_1\beta_1+\Lambda_2\beta_2$~(cf.~Section \ref{Real-spgp}).
    The dominance respects the compact positive root $(1,-1)$.
    This is noting but the fact that the set of equivalence classes of irreducible finite dimensional representations of $\mathrm{U}(2)\simeq K$ is parametrized by this set of dominant weights.
    For each dominant weight $\Lambda=(\Lambda_1,\Lambda_2)$, let $\tl$ be the irreducible representation of $K$ parametrized by $\Lambda$, which is the pullback of the irreducible representation $\det^{2\Lambda_2}\Sym^{\Lambda_1-\Lambda_2}$ of $\mathrm{U}(2)$ via the isomorphism $K\simeq \mathrm{U}(2)$.
    Here, $\Sym^{\Lambda_1-\Lambda_2}$ denotes the $(\Lambda_1-\Lambda_2)$-th symmetric tensor product representation of the two-dimensional standard representation of $\mathrm{U}(2)$. We also use $\tl$ to denote the infinitesimal action of $\tl$.

    We introduce a basis $\{Z,H,Y,Y'\}$ of $\fk_{\C}$ as follows:
    \[
    Z:=
    \begin{pmatrix}
    0_2 & -\sqrt{-1}\,\mathbf{1}_2\\
    \sqrt{-1}\,\mathbf{1}_2 & 0_2
    \end{pmatrix},~H:=\frac{1}{\sqrt{-1}}(T_1-T_2),~Y:=
    \begin{pmatrix}
    J_2 & 0_2\\
    0_2 & J_2
    \end{pmatrix},~Y':=
    \begin{pmatrix}
    0_2 & J'_2\\
    -J'_2 & 0_2
    \end{pmatrix}
    \]
    with $J_2:=
    \left(
    \begin{smallmatrix}
    0 & 1\\
    -1 & 0
    \end{smallmatrix}\right)$ and $J'_2=
    \left(
    \begin{smallmatrix}
    0 & 1\\
    1 & 0
    \end{smallmatrix}\right)$. 
    Then, the set $\{ H,X:=\frac{1}{2}(Y-\sqrt{-1}\,Y'),\overline{X}:=\frac{1}{2}(-Y-\sqrt{-1}\,Y')\}$ forms an ${\fraks\frakl}_2$-triple over $\bbC$. In fact, $[H,X]=2X,~[H,\overline{X}]=-2\overline{X},~[X,\overline{X}]=H$ holds. The element $Z$ belongs to the center. 
    
    We realize the representation $\tl$ on 
    \[
    V_{\Lambda}:=\{f\in\C[x_1,x_2]\mid \text{$f$ is homogeneous of degree $\Lambda_1-\Lambda_2$}\}.
    \]
    The representation space $V_{\Lambda}$ of $\tl$ has the dimension $d_{\Lambda}+1$ with $d_{\Lambda}=\Lambda_1-\Lambda_2$. We give two bases of $V_{\Lambda}$ as follows:
    \begin{align*}
    &\{v_i:=x_1^ix_2^{d_{\Lambda}-i}\mid 0\le i\le d_{\Lambda}\},\\
    &\{u_i:=(x_1+\sqrt{-1}\,x_2)^i(x_1-\sqrt{-1}\,x_2)^{d_{\Lambda}-i}\mid 0\le i\le d_{\Lambda}\}.
    \end{align*}
    The first basis satisfies
    \begin{align*}
    \tl(Z)v_k&=(\Lambda_1+\Lambda_2)v_k,\quad\tl(X)v_k=(d_{\Lambda}-k)v_{k+1},\\
    \tl(H)v_k&=(2k-d_{\Lambda})v_k,\quad\tl(\overline{X})v_k=kv_{k-1}
    \end{align*}
    for $0\le k\le d_{\Lambda}$.
    As for the second basis, we write down the following property
    \[
    \tl(Z')u_k=(2k-d_{\Lambda})u_k,~Z'=
    \begin{pmatrix}
    J_2 & 0_2\\
    0_2 & J_2
    \end{pmatrix}\in\fk
    \]
    for $0\le k\le d_{\Lambda}$. 
    
    Later, we will often need the contragredient representation $(\tl^*,V_{\Lambda}^*)$ with the representation space $V_{\Lambda}^*$.
    Its highest weight is $(-\Lambda_2,-\Lambda_1)$. 
    The space $V_{\Lambda}^*$ has a basis $\{v_k^*\}_{0\le k\le d_{\Lambda}}$ satisfying
    \begin{align*}
    \tl^*(Z)v_k^*&=(-\Lambda_1-\Lambda_2)v_k^*,\quad\tl^*(X)v_k^*=(k+1)v_{k+1}^*,\\
    \tl^*(H)v_k^*&=(2k-d_{\Lambda})v_k^*,\quad\tl^*(\overline{X})v_k^*=(d_{\Lambda}+1-k)v_{k-1}^*
    \end{align*}
    for $0\le k\le d_{\Lambda}$.
    This is nothing but what is obtained by replacing $(\Lambda_1,\Lambda_2)$ with $(-\Lambda_2,-\Lambda_1)$. As $\{u_i\mid 0\le i\le d_{\Lambda}\}$ is related to $\{v_i\mid 0\le i\le d_{\Lambda}\}$ we also provide a similar basis $\{u_i^*\mid 0\le i\le d_{\Lambda}\}$ of $V_{\Lambda}^*$ which is similarly related to $\{v_i^*\mid 0\le i\le d_{\Lambda}\}$ 
 and thus satisfies
     \begin{equation}\label{u-basis}
    \tl(Z')u^*_k=(2k-d_{\Lambda})u^*_k,~Z'=
    \begin{pmatrix}
    J_2 & 0_2\\
    0_2 & J_2
    \end{pmatrix}\in\fk
    \end{equation}
    for $0\le k\le d_{\Lambda}$.
    Note that $d_{\Lambda}$ remains the same under such replacement of the highest weights.
    
\subsection{Discrete series representations of $G$}\label{DS-rep}
    
    We provide Harish-Chandra's parametrization of discrete series representations for $G$, namely, irreducible unitary representations of $G$ whose matrix coefficients belong to $L^2(G)$.
    Our discussion is based on \cite[Theorem 16]{1966_Harish-Chandra} and \cite[Theorem 9.20,~Theorem 12.21]{2001_Knapp} with the help of \cite{1978_Kostant} and \cite{1978_Vogan}. 

    To parametrize such representations we need regular dominant analytically integral weights, namely, regular dominant weights coming from the derivatives of unitary characters of the compact Cartan subgroup $\exp({\mathfrak t})$ (see Section \ref{Real-spgp} for the notation ${\mathfrak t}$).
    Note that the unitary characters of $\exp({\mathfrak t})$ are parametrized by 
    \[
    \{a\beta_1+b\beta_2\mid (a,b)\in\Z^2\}\simeq\Z^2.
    \]
    We now see that the set of regular analytically integral weights is in bijection with 
    \[
    \Xi':=\{\lambda=(\lambda_1,\lambda_2)\in\Z^2\mid \text{$\langle\lambda,\beta\rangle\not=0$ for any $\beta\in\Delta^+$}\}.
    \] 
    Here, $\langle*,*\rangle$ denotes the inner product induced by the Killing form of ${\mathfrak g}$, which can be regarded as the standard inner product of the two-dimensional Euclidean space $\R^2$ in terms of the bijection $\{a\beta_1+b\beta_2\mid (a,b)\in\R^2\}\simeq\R^2$. 

    With the Harish-Chandra parametrization, we mean the classification of discrete series representations in terms of the infinitesimal equivalence~(i.e., the unitary equivalence of discrete series).
    Let $\calD_{\lambda}$ be the discrete series representation of $G$ parametrized by $\lambda\in\Xi'$.
    The representations $\calD_{\lambda}$ and $\calD_{\lambda'}$ are infinitesimally equivalent if and only if $\lambda$ and $\lambda'$ are conjugate by the Weyl group of $\fk$. 
    As a result, we see that the equivalence classes of discrete series representations of $G$ are in bijection with 
    \[
    \Xi:=\{\lambda\in\Xi'\mid \langle\lambda,(1,-1)\rangle>0\}=\{(\lambda_1,\lambda_2)\in\Xi'\mid \lambda_1>\lambda_2\},
    \]
    each of which is called a Harish-Chandra parameter~(cf.~\cite[Terminology after Theorem 9.20]{2001_Knapp}). 

    Now we introduce the following four sets of positive root system including the set $\Delta_{c}^+:=\{(1,-1)\}$ of the compact positive root:
    \begin{align*}
    \Delta_{I}^+&:=\{(2,0),~(0,2),~(1,1),~(1,-1)\},\\
    \Delta_{II}^+&:=\{((2,0),~(0,-2),~(1,1),~(1,-1)\},\\
    \Delta_{III}^+&:=\{(2,0),~(0,-2),~(-1,-1),~(1,-1)\},\\
    \Delta_{IV}^+&:=\{(-2,0),~(0,-2),~(-1,-1),~(1,-1)\}.
    \end{align*}
    Corresponding to the above four sets, we introduce the four sets of regular dominant weights as follows:
    \begin{align*}
    \Xi_{I}&:=\{(\lambda_1,\lambda_2)\in\Z_{>0}^2\mid\lambda_1>\lambda_2\},\\
    \Xi_{II}&:=\{(\lambda_1,\lambda_2)\in\Z_{>0}\times\Z_{<0}\mid\lambda_1>-\lambda_2\},\\
    \Xi_{III}&:=\{(\lambda_1,\lambda_2)\in\Z_{>0}\times\Z_{<0}\mid\lambda_1<-\lambda_2\},\\
    \Xi_{IV}&:=\{(\lambda_1,\lambda_2)\in\Z_{<0}^2\mid\lambda_1>\lambda_2\}.
    \end{align*}
    Note that $\Xi=\coprod_{*=I,II,III,IV}\Xi_{*}$.
    Discrete series representations of $G$ parametrized by $\Xi_{I}$~(resp.~$\Xi_{IV}$) are called holomorphic discrete series representations~(resp.~anti-holomorphic discrete series representations).
    Discrete series representations of $G$ parametrized by $\Xi_{II}\cup\Xi_{III}$ are called large in the sense of Vogan \cite[Section 6]{1978_Vogan}. 
    The large discrete series representations are characterized by the maximality of their Gelfand-Kirillov dimensions (or by having the minimal primitive ideals) among the discrete series representations.
    They are known to admit Whittaker models~(cf.~\cite[Theorem 6.8.1]{1978_Kostant}) and are therefore also called generic discrete series representations.
    In fact, they admit  rapidly decreasing Whittaker models~(cf.~Oda \cite{1994_Oda}).

    For $(\lambda_1,\lambda_2)\in\Xi$, the discrete series representation parametrized by $(-\lambda_2,-\lambda_1)$ is contragredient to that parametrized by $(\lambda_1,\lambda_2)$. We then see that the discrete series representations parametrized by $\Xi_{II}$ and $\Xi_{III}$ are in bijection with the contragredient representations of those parametrized by $\Xi_{III}$ and $\Xi_{II}$, respectively. 

    For each $\lambda\in\Xi$, recall that 
    \[
    \Lambda=(\Lambda_1,\Lambda_2):=\lambda+\rho_n-\rho_c
    \]
    with the half sum $\rho_n$~(resp.~$\rho_c$) of non-compact positive roots~(resp.~compact positive roots) is called the Blattner parameter~(cf.~\cite[Terminology after Theorem 9.20]{2001_Knapp}), which provides the highest weight of the minimal $K$-type~(cf.~\cite[pp.~626]{2001_Knapp}) of the discrete series representation $\calD_{\lambda}$. 
    We should point out that the minimal $K$-type is the most important multiplicity one $K$ type of the discrete series representation. 
    For $\lambda=(\lambda_1,\lambda_2)\in\Xi_{I}$ or $\Xi_{IV}$ we have $\Lambda=(\lambda_1+1,\lambda_2+2)$ or $(\lambda_1-2,\lambda_2-1)$, respectively. 
    When $\lambda=(\lambda_1,\lambda_2)$ is in $\Xi_{II}$ or $\Xi_{III}$, we have $\Lambda=(\lambda_1+1,\lambda_2)$ or $(\lambda_1,\lambda_2-1)$, respectively.

\section{Embedding of the large discrete series representations into parabolically induced representations of {$G$}}\label{section_emb_ind_rep}
    
    In this section, for later use, we discuss an embedding of large discrete series representations into generalized principal series representations.
    The socle series of such representations are determined by \cite{2009_Muic} for many cases.
    We discuss the remaining cases and apply the result to investigate the constant terms of automorphic forms in Section \ref{section_main}.
    
\subsection{Induction from $P_S$}
    
    We consider the following induced representations:
    \[
    \Ind_{P_S}^G \left(|\cdot|^{\frac{\lambda_1-\lambda_2}{2}} \otimes \calD_{\lambda_1+\lambda_2+1}\right), \qquad \Ind_{P_S}^G \left(|\cdot|^{\frac{\lambda_1+\lambda_2}{2}} \otimes \calD_{\lambda_1-\lambda_2+1}\right).
    \]
    Here, the symbol $|\cdot|$ means the absolute value of the determinant of $L_S(\bbR) = \GL_2(\bbR)$ and $\calD_k$ is the discrete series representation of $\GL_2(\bbR)$ with the trivial central character.
    Note that the representations $|\cdot|^{\frac{\lambda_1-\lambda_2}{2}} \otimes\calD_{\lambda_1+\lambda_2+1}$ and $|\cdot|^{\frac{\lambda_1+\lambda_2}{2}} \otimes \calD_{\lambda_1-\lambda_2+1}$ are isomorphic to $\delta(|\cdot|^{(\lambda_1+\lambda_2)/2}\sgn^{\lambda_2},\lambda_1-\lambda_2)$ and $\delta(|\cdot|^{(\lambda_1-\lambda_2)/2}\sgn^{\lambda_2},\lambda_1+\lambda_2)$, respectively, as in \cite{2009_Muic}.
    We also note that the large discrete series representation $D_{(\lambda_1,\lambda_2)}$ with the parameter $(\lambda_1,\lambda_2)$ of type II or III is the same as $X(\lambda_1,\lambda_2)$, as in \cite{2009_Muic}.
    The tensor product $\calD_k \otimes \sgn$ is isomorphic to $\calD_k$.

    \begin{lem}\label{emb_ps}
    \begin{enumerate}
    \item Let $\lambda=(\lambda_1,\lambda_2) \in \Xi_{II}$.
    The large discrete series representation $\calD_\lambda$ can then be embedded into
    \[
    \Ind_{P_S}^G \left(|\cdot|^{\frac{\lambda_1-\lambda_2}{2}} \otimes \calD_{\lambda_1+\lambda_2+1}\right)
    \isom 
    \Ind_{P_S}^G
    \left(|\cdot|^{\frac{\lambda_1-\lambda_2}{2}} \sgn \otimes \calD_{\lambda_1+\lambda_2+1}\right)
    \]
    and
    \[
    \Ind_{P_S}^G
    \left(|\cdot|^{\frac{\lambda_1+\lambda_2}{2}} \otimes \calD_{\lambda_1-\lambda_2+1}\right)
    \isom
    \Ind_{P_S}^G \left(|\cdot|^{\frac{\lambda_1+\lambda_2}{2}}\sgn \otimes \calD_{\lambda_1-\lambda_2+1}\right).
    \]
    \item Let $\lambda = (\lambda_1,\lambda_2) \in \Xi_{III}$.
    The large discrete series representation $\calD_\lambda$ can then be embedded into
    \[
    \Ind_{P_S}^G
    \left(|\cdot|^{\frac{\lambda_1-\lambda_2}{2}} \otimes \calD_{-\lambda_1-\lambda_2+1}
    \right)
    \isom
    \Ind_{P_S}^G
    \left(|\cdot|^{\frac{\lambda_1-\lambda_2}{2}} \sgn\otimes \calD_{-\lambda_1-\lambda_2+1}\right)
    \]
    and
    \[
    \Ind_{P_S}^G
    \left(|\cdot|^{-\frac{\lambda_1+\lambda_2}{2}} \otimes \calD_{\lambda_1-\lambda_2+1}
    \right)
    \isom
    \Ind_{P_S}^G
    \left(|\cdot|^{-\frac{\lambda_1+\lambda_2}{2}}\sgn \otimes \calD_{\lambda_1-\lambda_2+1}\right).
    \]
    \end{enumerate}
    \end{lem}
    \begin{proof}
    The statements follow directly from \cite[Theorem 10.1]{2009_Muic} and \cite[Theorem 10.3]{2009_Muic}.
    \end{proof}
    
    \subsection{Induction from $P_J$}
    \begin{lem}\label{emb_pj}
    \begin{enumerate}
        \item 
    Let $\lambda=(\lambda_1,\lambda_2) \in \Xi_{II}$.
    For characters $\mu_1,\mu_2 \in \{1,\sgn\}$ of $A_J(\bbR)=\bbR^\times$, we can embed the large discrete series representation $\calD_\lambda$ into
    \[
    \text{
    $
    \Ind_{P_J}^G
    \left(\mu_1|\cdot|^{-\lambda_2} \otimes \calD^+_{\lambda_1+1}\right),
    \qquad
    \left(
    \text{resp.~}
    \Ind_{P_J}^G
    \left(\mu_2|\cdot|^{\lambda_1} \otimes \calD^+_{-\lambda_2+1}\right)
    \right)
    $
    }
    \]
    if and only if $\mu_1=\sgn^{\lambda_2}$ (resp.~$\mu_2=\sgn^{\lambda_1}$).
    
    \item Let $\lambda = (\lambda_1,\lambda_2) \in \Xi_{III}$.
    For characters $\mu_1,\mu_2 \in \{1,\sgn\}$ of $A_J(\bbR)$, we can embed the large discrete series representation $\calD_\lambda$ into
    \[
    \text{$
    \Ind_{P_J}^G \left(\mu_1|\cdot|^{-\lambda_2} \otimes \calD^-_{\lambda_1+1}\right), \qquad 
    \left(
    \left(
    \text{resp.~}\Ind_{P_J}^G \mu_2|\cdot|^{\lambda_1} \otimes \calD^-_{-\lambda_2+1}
    \right)
    \right)
    $}
    \]
    if and only if $\mu_1=\sgn^{\lambda_2}$ (resp.~$\mu_2=\sgn^{\lambda_1}$).
    \end{enumerate}
    \end{lem}
    \begin{proof}
    We only prove the case of type II as the case of type III is similar.
    The ``if part'' follows from \cite[Theorem 10.1]{2009_Muic} and \cite[Theorem 11.2(ii)]{2009_Muic}.
    Suppose $\mu_1\neq\sgn^{\lambda_2}$ and $\mu_2\neq\sgn^{\lambda_1}$.
    Then, all the constituents of the induced representations are non-temperd by \cite[Lemma 9.4]{2009_Muic}, for which the discrete series representations are tempered.
    \end{proof}
    
    For the case of $P_0$, we need to investigate the degenerate Whittaker functions associated with the trivial character.
    We give the statements in \S \ref{section_min_parab}.
    
    \section{Generalized Whittaker functions for degenerate characters of unipotent radicals}\label{section_Whitt}
    
    Both the Whittaker functions and the Fourier-Jacobi type spherical functions taken up in this section are important ingredients of the archimedean aspect of our paper.
    In contrast to Whittaker functions in the usual sense, they are attached to degenerate characters of $N_0$ and $N_J$.
    We call these generalized Whittaker functions and provide their explicit formulas.

\subsection{Whittaker functions for degenerate characters of $N_0$}\label{Deg-Whitt}
    
    Recall that $N_0$ denotes the unipotent radical of the minimal parabolic subgroup $P_0$~(cf.~Section \ref{Real-spgp}).
    Unitary characters of $N_0$ are of the form
    \[
    N_0\ni n(u_0,u_1,u_2,u_3)\mapsto\exp \left(2\pi\sqrt{-1}(c_0u_0+c_3u_3)\right)\in\C^1
    \]
    with $(c_0,c_3)\in\R^2$.
    For a fixed unitary character $\psi$ of $N_0$ and an irreducible representation $\tau$ of $K$, we introduce the following two spaces of $C^\infty$ functions:
    \begin{align*}
    C^{\infty}_{\psi}(N_0\backslash G))&:=\{f:G \rightarrow \bbC\mid \text{$f(ng)=\psi(n)f(g)$ for any $(n,g)\in N_0\times G$}\},\\
    C^{\infty}_{\psi,\tau}(N_0\backslash G/K)&:=\{w:G \rightarrow  V_\tau \mid \text{$w(ngk)=\psi(n)\tau(k)^{-1}w(g)$ for any $(n,g,k)\in N_0\times G\times K$}\}.
    \end{align*}
    Here, $V_{\tau}$ denotes the representation space of $\tau$.
    We then define the space of Whittaker functions for a large discrete series representation $\calD_{\lambda}$ with $K$-type $\tau$ by the image of the restriction map
    \[
    \Hom_{(\g,K)}(\calD_{\lambda},C^{\infty}_{\psi}(N_0\backslash G))\ni\phi\mapsto\phi\circ\iota\in\Hom_K(\tau,C^{\infty}_{\psi}(N_0\backslash G))
    \simeq
    C^{\infty}_{\psi,\tau^*}(N_0\backslash G/K)
    \]
    to the $K$-type $\tau$, where $\iota:\tau\hookrightarrow\calD_{\lambda}$ is a $K$-inclusion and $\tau^*$ denotes the contragredient of $\tau$.
    Note here that the multiplicity of $\tau$ in $\calD_{\lambda}$ should be one to ensure the well-defined notion of the Whittaker functions. 
    The Whittaker functions in the usual sense are attached to non-degenerate characters of $N_0$, i.e., $c_0c_3\not=0$.
    Here, we are interested in Whittaker functions for degenerate unitary characters with
    \[
    (i)~c_0\not=0,~c_3=0\quad(ii)~c_0=c_3=0.
    \]
    We now take $\tau$ to be the minimal $K$-type $\tl$ of $\calD_{\lambda}$, and let $\Phi(g):=\sum_{i=0}^{d_{\Lambda}}\phi_i(g)v_i^*$ be the Whittaker function for $\calD_{\lambda}$ with $K$-type $\tl$ and $d_{\Lambda}:=\Lambda_1-\Lambda_2$ (cf.~Section \ref{RepMaxCpt}).
    We write down the differential equations characterizing the Whittaker functions.
    They arise from the infinitesimal action of the Dirac-Schmid operator.
    The following lemma is obtained by a direct calculation of \cite[Lemma (6.2) (ii)]{1994_Oda}, for which note that the Whittaker functions are defined for contragredient representations of large discrete series representations in \cite{1994_Oda}.
    
    \begin{lem}\label{Diffeq-basic}
    Let $\lambda\in\Xi_{II}$. The Whittaker function $\Phi$ for $\calD_{\lambda}$ with the minimal $K$-type $\tl$ satisfies the following system of the differential equations:
    \begin{align*}
        &R(X_{(2,0)})\phi_i(g)-R(X_{(1,1)})\phi_{i+1}(g)+R(X_{(0,2)})\phi_{i+2}(g)
        =0,\\
        &R(X_{(0,-2)})\phi_i(g)+R(X_{(-1,-1)})\phi_{i+1}(g)+R(X_{(-2,0)})\phi_{i+2}(g)
        =0,\\
        &2jR(X_{(0,-2)})\phi_{j-1}(g)-(d_{\Lambda}-2j)R(X_{(-1,-1)})\phi_j(g)-2(d_{\Lambda}-j)R(X_{(-2,0)})\phi_{j+1}(g)
        =0,
    \end{align*}
    for any $0\le i\le d_{\Lambda}-2$ and $0\le j\le d_{\Lambda}$.
    Here, $R$ denotes the derivative of the right translation by $G$.
    \end{lem}
    
    We also will need the following lemma (cf.~\cite[Section  4]{2021_Narita}):
    \begin{lem}\label{Whittaker-MVWinv}
    Whittaker functions $W$ for $\lambda\in\Xi_{III}$ are related to those for $\lambda\in\Xi_{II}$ by
    \[
    W(\delta g\delta^{-1}\xi)
    \]
    and vice versa, 
    where $\delta:=
    \left(
    \begin{smallmatrix}
    -\mathbf{1}_2 & 0_2\\
    0_2 & \mathbf{1}_2
    \end{smallmatrix}
    \right),
    ~\xi:=
    \left(
    \begin{smallmatrix}
    J'_2 & 0_2\\
    0_2 & J'_2
    \end{smallmatrix}
    \right)$ with 
    $J'_2:=
    \left(
    \begin{smallmatrix}
    0 & 1\\
    1 & 0
    \end{smallmatrix}\right)$. 
    \end{lem}
    The above lemma could be understood in terms of the Moeglin-Vigner{\'a}-Waldspurger involution (cf.~\cite{1987_MVW},~\cite{2019_Prasad} et al.) or the real Chevalley involution (cf.~\cite{2014_Adams} et al.), which relates an irreducible admissible representation to its contragredient.
    
    The Whittaker functions are determined by their restriction to $A_0^\infty$ in view of their left $N_0$-equivalence and right $K$-equivalence. 
    We will explicitly write down the differential equations of the Whittaker functions restricted to $A_0^\infty$, also known as the restriction to the radial part.
    The results are stated as Theorems \ref{Deg-Whitt-Siegel} and \ref{Deg-Whitt-Borel}.
    We will give their proofs based on an article being prepared by Taku Ishii and the second named author.
    These deal with Whittaker functions on $G$ attached to degenerate characters of $N_0$ in a general setting.
    To be precise about the two theorems, we do some reduction of the argument for their proofs by Lemma \ref{Whittaker-MVWinv}, and Theorem \ref{Deg-Whitt-Borel} is what was modified by the first named author pointing out the missing ``$f_{0,i}$" in its statement.
    
\subsection*{(i)~The case of $c_0\not=0,~c_3=0$}
    
    To begin, we provide a system of differential equations characterizing the Whittaker functions for this case $c_0\not=0,~c_3=0$.
    To this end, we use the basis $\{u_i^*\mid 0\le i\le d_{\Lambda}\}$ of $V_{\Lambda}^*$~(cf.~Section \ref{RepMaxCpt}). 
    To know the Whittaker functions in detail for this case, this basis is more useful than the basis $\{v_i^*\mid 0\le i\le d_{\Lambda}\}$~(cf.~Section \ref{RepMaxCpt}).
    One reason is each $u_i^*$ is an eigenvector with respect to a maximal compact subgroup of $M_S$ isogenous to $\SO(2)$, as the transformation formula of $u_i^*$ in Section \ref{RepMaxCpt} (\ref{u-basis}) indicates.
    This property is necessary in the discussion of Section \ref{RepWhittFct} (i).
    
    \begin{prop}\label{DiffEq-deg1}
    For $\lambda\in\Xi_{II}$, let $\sum_{i=0}^{d_{\Lambda}}\varphi_i(g)u_i^*$ be a Whittaker function of the case $c_0\not=0,~c_3=0$ for $\calD_{\lambda}$ with minimal $K$-type $\tl$, where $\{u_i^*\mid 0\le i\le d_{\Lambda}\}$ denotes the basis of $V_{\Lambda}^*$ as above. 
    Let $a_0:=\diag(a_1,a_2,a_1^{-1},a_2^{-1})\in A_0^\infty$. Put $\partial_i:=a_i\frac{\partial}{\partial a_i}$ for $i=1,2$.
    \begin{align}
        \left(\partial_1-\partial_2-\Lambda_1+\Lambda_2+2i-4\pi c_0\frac{a_1}{a_2}\right)\varphi_i(a_0)
        -2\left(\Lambda_1+\Lambda_2\right)\varphi_{i+1}(a_0)\qquad&\notag\\
        +\left(\partial_1-\partial_2+\Lambda_1-\Lambda_2-2i-4+4\pi c_0\frac{a_1}{a_2}\right)\varphi_{i+2}(a_0)&=0,\nonumber\tag{$A_i$}\\
        (\partial_1+\partial_2-\Lambda_1+\Lambda_2-2)\varphi_{i+1}(a_0)&=0 \tag{$B_i$},
    \end{align}
    for $0\le i\le d_{\Lambda}-2$ and 
    \begin{align}
        &i\left(-\partial_1+\partial_2+4\pi c_0\frac{a_1}{a_2}+\Lambda_1-\Lambda_2-2i+2\right)\varphi_{i-1}(a_0) \tag{$C_i$}\\
        &\qquad+\left(\Lambda_1-\Lambda_2-2i\right)(\partial_1+\partial_2-\Lambda_1-\Lambda_2-2)\varphi_i(a_0) \nonumber\\
        &\qquad\qquad+(\Lambda_1-\Lambda_2-i)\left(\partial_1-\partial_2+4\pi c_0\frac{a_1}{a_2}+\Lambda_1-\Lambda_2-2i-2\right)\varphi_{i+1}(a_0)=0,\nonumber
    \end{align}
    for $0\le i\le d_{\Lambda}$. 
    \end{prop}
    
    \begin{proof}
    We can write the differential equations characterizing Whittaker functions $\sum_{i=0}^{d_{\Lambda}}\phi_i(g)v_i^*$ with respect to $\{v_i^*\mid 0\le i\le d_\Lambda\}$ by plugging the Iwasawa decomposition of non-compact root vectors~(cf.~Lemma \ref{Iwasawa-decomp}) into the formulas in Lemma \ref{Diffeq-basic} and by the left and right equivariance with respect to $\psi$ and $\tl^*$. Here, regarding the right $\tl^*$-equivariance, note the formula for $\tl^*$ around the end of Section \ref{RepMaxCpt}. These are given as follows:
    \begin{align}
    (\partial_1+\Lambda_2-i-2)\phi_i(g)-4\pi\sqrt{-1}\,c_0\frac{a_1}{a_2}\phi_{i+1}(g)+(\partial_2+\Lambda_1-i-2)\phi_{i+2}(g)&=0,\tag{$A'_i$}\\
    (\partial_2-\Lambda_1+i)\phi_i(g)+4\pi\sqrt{-1}\,c_0\frac{a_1}{a_2}\phi_{i+1}(g)+(\partial_1-2\Lambda_1+\Lambda_2+i)\phi_{i+2}(g)&=0, \tag{$B'_i$}\\
    j(\partial_2-\Lambda_1+j-1)\phi_{j-1}(g)-2\pi\sqrt{-1}\,c_0(\Lambda_1-\Lambda_2-2j)\frac{a_1}{a_2}\phi_{j}(g)\qquad& \tag{$C'_j$}\\
    -(\Lambda_1-\Lambda_2-j)(\partial_1-\Lambda_1+j-1)\phi_{i+1}(g)&=0.\nonumber
    \end{align}
    Here, $i$ runs over $0\le i\le d_{\Lambda}-2$ and $j$ runs over $0\le j\le d_{\Lambda}$.
    We rewrite these differential equations in terms of the basis $\{u_i^*\mid 0\le i\le d_{\Lambda}\}$ of $V_{\Lambda}^*$. 
    To this end we use the following lemma:
    
    \begin{lem}\label{Beta-lemma}
    Let $n$ be a positive integer. For $1\le i,j\le n$ we define $\beta_{ij}^n\in\C$ by 
    \begin{equation}\label{definition-beta}
    (x_1+\sqrt{-1}\,x_2)^i(x_1-\sqrt{-1}\,x_2)^{n-i}=\sum_{j=0}^n\beta_{ij}^nx_1^jx_2^{n-j}.
    \end{equation}
    We put $h_i:=\varphi_i(g)$ and $f_j:=\phi_i(g)$, where $\varphi_i$ and $\phi_i$ denote the coefficient functions of a Whittaker function $\Phi$ with respect to the basis $\{u^*_i\mid 0\le i\le d_{\Lambda}\}$ and $\{v^*_j\mid 0\le j\le d_{\Lambda}\}$ of $V_{d_{\Lambda}}$, respectively.
    We then have $h_i=\sum_{j=0}^n\beta_{ij}^nf_j$ and the following formula:
    \begin{enumerate}
    \item 
    $
    \sum_{j=0}^{d_{\Lambda}}\beta_{ij}^{d_{\Lambda}} (j(j-1)f_{j-1}+(d_{\Lambda}-j)(d_{\Lambda}-j-1)f_{j+1})
    =
    \sqrt{-1}\,i(d_{\Lambda}-i)(h_{i-1}-h_{i+1})
    $,
    \item 
    $
        \sum_{j=0}^{d_{\Lambda}}j\beta_{ij}^{d_{\Lambda}}f_{j-1}
    =
    \frac{\sqrt{-1}}{2}(ih_{i-1}+(d_{\Lambda}-i)h_{i}-(d_{\Lambda}-2i)h_{i+1})
    $,
    \item $\sum_{j=0}^{d_{\Lambda}}\beta_{ij}^{d_{\Lambda}}(d_{\Lambda}-2j)f_j=-(d_{\Lambda}-i)h_{i+1}-ih_{i-1}
    $,
    \item $\sum_{j=0}^{d_{\Lambda}}(d_{\Lambda}-j)\beta_{ij}^{d_{\Lambda}}f_{j+1}=\frac{\sqrt{-1}}{2}(ih_{i-1}-(d_{\Lambda}-2i)h_i-(d_{\Lambda}-i)h_{i+1})$,
    \item $\sum_{j=0}^{d_{\Lambda}-2}\beta_{ij}^{d_{\Lambda}-2}f_j=-\frac{1}{4}(h_i+h_{i+2})+\frac{1}{2}h_{i+1},~\sum_{j=0}^{d_{\Lambda}-2}\beta_{ij}^{d_{\Lambda}-2}f_{j+2}=\frac{1}{4}(h_i+h_{i+2})+\frac{1}{2}h_{i+1}$,
    \item $\sum_{j=0}^{d_{\Lambda}-2}\beta_{ij}^{d_{\Lambda}-2}f_{j+1}=\frac{1}{4\sqrt{-1}}(h_{i+2}-h_i)$,
    \item 
    $
    \sum_{j=0}^{d_{\Lambda}-2}\beta_{ij}^{d_{\Lambda}-2}
    (-jf_j+(d_{\Lambda}-j-2)f_{j+2})
    =
    \frac{1}{4}(-d_{\Lambda}+2i+2)(h_{i+2}-h_{i})
    $.
    \end{enumerate}
    \end{lem}
    
    \begin{proof}
    All the formulas are verified by direct calculations. 
    Regarding the first four, it is useful to consider the infinitesimal action of the differential operators  $x_1x_2(\frac{\partial^2}{\partial x_1^2}+\frac{\partial^2}{\partial x_2^2}),~x_2\frac{\partial}{\partial x_1},~x_2\frac{\partial}{\partial x_1}-x_1\frac{\partial}{\partial x_2}$ and $x_1\frac{\partial}{\partial x_2}$ on both sides of the equation $(\ref{definition-beta})$. 
    The four formulas just mentioned are deduced from the four equations thus obtained.
    \end{proof}
    
    The formula $(A_i)$ is verified by applying the fifth to seventh formulas in Lemma \ref{Beta-lemma} to
    \[
    \sum_{j=0}^{d_{\Lambda}-2}\beta_{ij}^{d_{\Lambda}-2}((A'_j)-(B'_j)).
    \]
    The formula $(B_i)$ is also verified by applying the formulas from fifth in Lemma \ref{Beta-lemma} to 
    \[
    \sum_{j=0}^{d_{\Lambda}-2}\beta_{ij}^{d_{\Lambda}-2}((A'_j)+(B'_j)).
    \]
    In fact, to verify the formula $(A_i)$, it is helpful to note that $(A'_j)-(B'_j)$ is equivalent to
    \begin{align*}
        &(\partial_1-\partial_2-2)(\phi_j(g)-\phi_{j+2}(g))-8\pi\sqrt{-1}c_0\frac{a_1}{a_2}\phi_{j+1}(g)\\
        &\qquad+(\Lambda_1+\Lambda_2)(\phi_j(g)+\phi_{j+2}(g))+2\{-j\phi_j(g)+(d_{\Lambda}-j-2)\phi_{j+2}(g)\}=0.
    \end{align*}
    The formula $(C_i)$ is obtained by applying the first four formulas in Lemma \ref{Beta-lemma} to $\sum_{j=0}^{d_{\Lambda}}\beta_{ij}^{d_{\Lambda}}(C'_j)$.
    \end{proof}
    
    We further need the confluent hypergeometric function $W_{\kappa, \mu}(y)$ characterized by the following differential equation:
    \[
    \frac{d^2}{dy^2}W_{\mu,\kappa}(y)
    +
    \left(
    -\frac{1}{4}+\frac{\kappa}{y}+\frac{1/4-\mu^2}{y^2}\right)
    W_{\kappa,\mu}(y)=0.
    \]
    If there is no fear of confusion, this can be called a Whittaker function, which is often done.
    For later use we give a well-known formula on $W_{\kappa,\mu}$ as
    follows~\cite[Section 3.6]{2011_Goldfeld_Hundley}:
    \begin{lem}\label{Whittaker-formula}
    \begin{align}
    \left(y\frac{d}{dy}-\frac{1}{2}y+\kappa\right)W_{\kappa,\mu}(y)&=-W_{\kappa+1,\mu}(y) \label{f1},\\
    \left(y\frac{d}{dy}+\frac{1}{2}y-\kappa\right)W_{\kappa,\mu}(y)&=-\left(\mu^2-\left(\kappa-\frac{1}{2}\right)^2\right)W_{\kappa-1,\mu}(y) \label{f2},\\
    W_{\mu+\frac{1}{2},\mu}(y)&=y^{\mu+\frac{1}{2}}\exp\left(-\frac{y}{2}\right),\label{f3}\\
    W_{\kappa,\mu}(y)&=W_{\kappa,-\mu}(y) \label{f4}.
    \end{align}
    \end{lem}
    Here, $W_{\kappa,\mu}$ means the solution of moderate growth for this differential equation.
    
    \begin{thm}[Ishii-Narita]\label{Deg-Whitt-Siegel}
    Suppose that $\psi$ satisfies $c_0\not=0,~c_3=0$.
    For a large discrete series representation $\calD_{\lambda}$ with Harish-Chandra parameter $\lambda$, we keep the notation $\sum_{i=0}^{d_{\Lambda}}\varphi_i(g)u_i^*$ of a Whittaker function for $\calD_{\lambda}$ with minimal $K$-type $\tl$, where $\{u_i^*\mid 0\le i\le d_{\Lambda}\}$ denotes the basis of $V_{\Lambda}^*$ given in Section \ref{RepMaxCpt}.
    We assume that this function is of moderate growth.
    \begin{enumerate}
    \item Let $\lambda\in\Xi_{II}$.  With arbitrary constants $C_0$ and $C_1$ independent of $i$, the restriction of coefficient functions $\varphi_i$ to $A_0^\infty$ is explicitly given as 
    \[
    C_0\alpha_i(a_1a_2)^{\frac{d_{\Lambda}+2}{2}}W_{{\rm sgn}(c_0)\left(i-\frac{d_{\Lambda}}{2}\right),\frac{\Lambda_1+\Lambda_2-1}{2}}\left(4\pi|c_0|\frac{a_1}{a_2}\right)+C_1\beta_ia_1^{\Lambda_1+1}a_2^{\Lambda_2+1}\exp\left(-2\pi|c_0|\frac{a_1}{a_2}\right),
    \]
    where ${\rm sgn}(c_0)$ denotes the signature of $c_0$. 
    Here,
    \[
    \alpha_i:=
    \begin{cases}
    0&(0\le i\le \Lambda_1-1)\\
    \frac{1}{(i-\Lambda_1)!}&(\Lambda_1\le i\le d_{\Lambda})
    \end{cases},\quad
    \beta_i:=\delta_{i,d_{\Lambda}}~(0\le i\le d_{\Lambda})
    \]
    when $c_0>0$, and
    \[
    \alpha_i:=
    \begin{cases}
    \frac{1}{(-\Lambda_2-i)!}&(0\le i\le -\Lambda_2)\\
    0&(-\Lambda_2+1\le i\le d_{\Lambda})
    \end{cases},\quad
    \beta_i:=\delta_{i,0}~(0\le i\le d_{\Lambda})
    \]
    when $c_0<0$.
    \item Let $\lambda\in\Xi_{III}$.
    The restriction of coefficient functions $\varphi_i$ to $A_0^\infty$ is explicitly given as 
    \[
    C_0\alpha_i(a_1a_2)^{\frac{d_{\Lambda}+2}{2}}W_{{\rm sgn}(c_0)\left(i-\frac{d_{\Lambda}}{2}\right),\frac{\Lambda_1+\Lambda_2+1}{2}}\left(4\pi|c_0|\frac{a_1}{a_2}\right)+C_1\beta_ia_1^{-\Lambda_2+1}a_2^{-\Lambda_1+1}\exp\left(-2\pi|c_0|\frac{a_1}{a_2}\right)
    \]
    with arbitrary constants $C_0$ and $C_1$ independent of $i$, 
    where
    \[
    \alpha_i:=
    \begin{cases}
    \frac{(-1)^i}{(i+\Lambda_2)!}&(0\le i\le -\Lambda_2)\\
    0&(-\Lambda_2+1\le i\le d_{\Lambda})
    \end{cases},\quad
    \beta_i:=(-1)^i\delta_{i,d_{\Lambda}}~(0\le i\le d_{\Lambda})
    \]
    when $c_0>0$, and
    \[
    \alpha_i:=
    \begin{cases}
    \frac{(-1)^i}{(\Lambda_1-i)!}&(0\le i\le \Lambda_1)\\
    0&(\Lambda_1+1\le i\le d_{\Lambda})
    \end{cases},\quad
    \beta_i:=(-1)^i\delta_{i,0}~(0\le i\le d_{\Lambda})
    \]
    when $c_0<0$.
    \end{enumerate}
    \end{thm}
    
    \begin{proof}
    In view of Lemma \ref{Whittaker-MVWinv}, it suffices to consider the case of $\lambda\in\Xi_{II}$. 
    Let us now solve the system of the differential equations in Proposition \ref{DiffEq-deg1}. 
    The calculation of $i(A_{i-1})+2i(B_{i-1})+(C_i)$ leads to
    \begin{align}
        (\partial_1+\partial_2-\Lambda_1-\Lambda_2-2i-2)\varphi_i(a_0)
        +\left(\partial_1-\partial_2+4\pi c_0\frac{a_1}{a_2}+\Lambda_1-\Lambda_2-2i-2\right)\varphi_{i+1}(a_0)=0 \tag{$D_i$}
    \end{align}
    for any $0\le i\le d_{\Lambda}-1$.
    The differential equations $(C_i)$ can be thus replaced by $(D_i)$ together with $(C_{d_{\Lambda}})$. 
    Let us consider $(A_i)-(D_{i+1})$, which is
    \begin{align}
        \left(\partial_1-\partial_2+\Lambda_1-\Lambda_2+2i-4\pi c_0\frac{a_1}{a_2}\right)\varphi_i(a_0)-(\partial_1+\partial_2++\Lambda_1+\Lambda_2-2i-4)\varphi_{i+1}(a_0)=0\tag{$E_i$}
    \end{align}
    for any $0\le i\le d_{\Lambda}-2$.
    By calculating 
    \[
    \left(\partial_1-\partial_2-\Lambda_1+\Lambda_2+2i-4\pi c_0\frac{a_1}{a_2}\right)(D_i)+\left(\partial_1+\partial_2-\Lambda_1-\Lambda_2-2i-2\right)(E_i)
    \]
    and taking $(B_i)$ into account, we obtain
    \begin{equation}
        \left((\partial_1-2)^2+(\partial_2-1)^2-(\Lambda_1-1)^2-\Lambda_2^2+4(-\Lambda_1+\Lambda_2+2i)\pi c_0\frac{a_1}{a_2}-8\left(\pi c_0\frac{a_1}{a_2}\right)^2\right)\varphi_i(a_0)=0 \tag{$F_i$}
    \end{equation}
    for $1\le i\le d_{\Lambda}-1$.
    By
    \[
    \left(\partial_1+\partial_2+\Lambda_1+\Lambda_2+2i-4\right)(D_i)
    +\left(\partial_1-\partial_2-4\pi c_0\frac{a_1}{a_2}+\Lambda_1-\Lambda_2-2i-2\right)(E_i),
    \]
    we see that $(F_i)$ holds also for $i=0$.
    
    We now carry out the change of variables $(y_1,y_2):=\left(4\pi c_0\frac{a_1}{a_2},a_1a_2\right)$.
    The differentials $(\partial_1,\partial_2)$ are then changed into $(\partial_{y_1}+\partial_{y_2},-\partial_{y_1}+\partial_{y_2})$ with $\partial_{y_i}:=y_i\frac{\partial}{\partial y_i}$ for $i=1,2$. 
    From $(F_i)$ and $(B_i)$, we deduce
    \[
    \varphi_i(a_0)
    =
    C_iy_2^{\frac{\Lambda_1-\Lambda_2+2}{2}}W_{{\rm sgn}(c_0)\left(i-\frac{d_{\Lambda}}{2}\right),\frac{\Lambda_1+\Lambda_2-1}{2}}(y_1)\quad(1\le i\le d_{\Lambda}-1).
    \]
    From $(D_i)$ and $(C_{d_{\Lambda}})$, for $0\le i\le d_{\Lambda}-1$, we deduce
    \begin{align*}
        \left(\partial_{y_2}-\frac{\Lambda_1+\Lambda_2}{2}-i-1\right)\varphi_i(a_0)+\left(\partial_{y_1}+\frac{1}{2}{\rm sgn}(c_0)y_1+\frac{\Lambda_1-\Lambda_2}{2}-i-1\right)\varphi_{i+1}(a_0)
        &=0,\\
        \left(\partial_{y_1}-\frac{1}{2}{\rm sgn}(c_0)y_1+\frac{\Lambda_1-\Lambda_2-2}{2}\right)\varphi_{d_{\Lambda}-1}(a_0)+\left(\partial_{y_2}-\frac{\Lambda_1+\Lambda_2}{2}-1\right)\varphi_{d_{\Lambda}}(a_0)
        &=0.
    \end{align*}
    
    To go further, we divide the situation into two cases $c_0>0$ and $c_0<0$. 
    We discuss only the first case here since the argument of the second one is similar. 
    
    We now suppose that $c_0>0$.
    In view of (\ref{f1}) and (\ref{f2}), for $1\le i\le d_{\Lambda}-2$, we obtain
    \begin{align*}
        C_i-(i+1-\Lambda_1)C_{i+1}
        &=0,\\
        \left(\partial_{y_2}-\frac{\Lambda_1+\Lambda_2}{2}-1\right)\varphi_0(a_0)-C_1\Lambda_2(\Lambda_1-1)y_2^{\frac{\Lambda_1-\Lambda_2+2}{2}}W_{-\frac{\Lambda_1-\Lambda_2}{2},\frac{\Lambda_1+\Lambda_1-1}{2}}(y_1)
        &=0,\\
        (-\Lambda_1+1)C_{d_{\Lambda}-1}y_2^{\frac{d_{\Lambda}+2}{2}}W_{\frac{\Lambda_1-\Lambda_2-2}{2}}(y_1)
        +\left(\partial_{y_1}+\frac{1}{2}y_1-\frac{\Lambda_1-\Lambda_2}{2}\right)\varphi_{d_{\Lambda}}(a_0)
        &=0,\\
        -C_{d_{\Lambda}-1}y_2^{\frac{\Lambda_1-\Lambda_2+2}{2}}W_{\frac{\Lambda_1-\Lambda_2}{2},\frac{\Lambda_1+\Lambda_2-1}{2}}(y_1)+\left(\partial_{y_2}-\frac{\Lambda_1+\Lambda_2}{2}-1\right)\varphi_{d_{\Lambda}}(a_0)
        &=0
    \end{align*}
    from $(E_i)$ and $(C_{d_{\Lambda}})$.
    From the second equation above and $(E_0)$, we get
    \[
    \varphi_0(a_0)=C_0y_2^{\frac{d_{\Lambda}+2}{2}}W_{-\frac{d_{\Lambda}}{2},\frac{\Lambda_1+\Lambda_2-1}{2}}(y_1).
    \]
    For this, we note that, although we encounter another candidate $y_1^{\frac{\Lambda_1-\Lambda_2}{2}}y_2^{\frac{\Lambda_1+\Lambda_2+2}{2}}e^{\frac{1}{2}y_1}$ of the solution in the course of the calculation, we can exclude this by the moderate growth condition on $\varphi_i$s. 
    From the third and fourth equations above we deduce
    \[
    \varphi_{d_{\Lambda}}(a_0)=C'_{d_{\Lambda}}y_1^{\frac{\Lambda_1-\Lambda_2}{2}}y_2^{\frac{\Lambda_1+\Lambda_2}{2}+1}e^{-\frac{1}{2}y_1}-\frac{C_{d_{\Lambda}-1}}{\Lambda_2}y_2^{\frac{d_{\Lambda}+2}{2}}W_{\frac{d_{\Lambda}}{2},\frac{\Lambda_1+\Lambda_2-1}{2}}(y_1)
    \]
    with another arbitrary constant $C'_{d_{\Lambda}}$ dependent only on $d_{\Lambda}$. 
    The recurrence relation of the constants $C_i$ in the 1st formula above implies
    \[
    C_{i}=
    \begin{cases}
    0&(0\le i\le \Lambda_1-1)\\
    \frac{(-\Lambda_2)!}{(i-\Lambda_1)!}C_{d_{\Lambda}}&(\Lambda_1\le i\le d_{\Lambda})
    \end{cases}.
    \]
    As a result, we are able to summarize the argument so far to obtain the result in the assertion.
    \end{proof}

\subsection*{(ii)~The case of $c_0=c_3=0$}

    \begin{thm}[Horinaga-Ishii-Narita]\label{Deg-Whitt-Borel}
    Suppose that $c_0=c_3=0$.
    For a large discrete series representation $\calD_{\lambda}$ with Harish-Chandra parameter $\lambda$, we denote a Whittaker function of $\calD_{\lambda}$ with the minimal $K$-type $\tl$ of $\calD_{\lambda}$ by $\sum_{i=0}^{d_{\Lambda}}\phi_i(g)v_i^*$.
    \begin{enumerate}
    \item Let $\lambda\in\Xi_{II}$. The restriction of coefficient functions $\phi_i$ to $A_0^\infty$ is explicitly given as
    \[
    C_0f_{0,i}(a_0)+C_1f_{1,i}(a_0)+C_2f_{2,i}(a_0)+C_3f_{3,i}(a_0)+C_4f_{4,i}(a_0)
    \]
    with arbitrary constants $C_0,~C_1,~C_2,~C_3$ and $C_4$, where
    \begin{align*}
    f_{0,i}(a_0)&:=
    \begin{cases}
    a_1^{2-\Lambda_2}a_2^{\Lambda_1}&(i=0\\
    0&(0< i\le d_{\Lambda})
    \end{cases},\\
    f_{1,i}(a_0)&:=
    \begin{cases}
    (-1)^{i/2}a_1^{\Lambda_1+1}a_2^{\Lambda_2+1}&(\text{$i$:even})\\
    0&(\text{$i$:odd})
    \end{cases},\\
    f_{2,i}(a_0)&:=
    \begin{cases}
    0&(\text{$i$:even})\\
    (-1)^{\frac{i-1}{2}}a_1^{\Lambda_1+1}a_2^{\Lambda_2+1}&(\text{$i$:odd})
    \end{cases},\\
    f_{3,i}(a_0)&:=
    \begin{cases}
    2^{i/2}(\frac{-d_{\Lambda}+2}{2})_{i/2}a_1^{ \Lambda_1+1}a_2^{-\Lambda_2+1}&(\text{$i$ is even and $0\le i\le d_{\Lambda}+\delta(2\Lambda_2-1)$})\\
    0&(\text{otherwise})
    \end{cases},\\
    f_{4,i}(a_0)&:=
    \begin{cases}
    2^{\frac{i-1}{2}}(\frac{-d_{\Lambda}+2}{2})_{\frac{i-1}{2}}a_1^{ \Lambda_1+1}a_2^{-\Lambda_2+1}&(\text{$i$ is odd and $0\le i\le d_{\Lambda}+(1-\delta)(2\Lambda_2-1)$})\\
    0&(\text{otherwise})
    \end{cases}.
    \end{align*}
    Here, $\delta\in\{0,1\}$ is determined by $d_{\Lambda}\equiv\delta\mod 2$.
    \item Let $\lambda\in\Xi_{III}$. The restriction of coefficient functions $\varphi_i(g)$ to $A_0^\infty$ is explicitly given as
    \[
    C_0f_{0,i}(a_0)+C_1f_{1,i}(a_0)+C_2f_{2,i}(a_0)+C_3f_{3,i}(a_0)+C_4f_{4,i}(a_0)
    \]
    with arbitrary constants $C_0,~C_1,~C_2,~C_3$, and $C_4$, where
    \begin{align*}
    f_{0,i}(a_0)&:=
    \begin{cases}
    a_1^{2+\Lambda_1}a_2^{-\Lambda_2}&(i=d_{\Lambda})\\
    0&(0\le i<d_{\Lambda})
    \end{cases},\\
    f_{1,i}(a_0)&:=
    \begin{cases}
    (-1)^{i/2}a_1^{-\Lambda_2+1}a_2^{-\Lambda_1+1}&(\text{$i$:even})\\
    0&(\text{$i$:odd})
    \end{cases},\\
    f_{2,i}(a_0)&:=
    \begin{cases}
    0&(\text{$i$:even})\\
    (-1)^{\frac{i-1}{2}}a_1^{-\Lambda_2+1}a_2^{-\Lambda_1+1}&(\text{$i$:odd})
    \end{cases},\\
    f_{3,i}(a_0)&:=
    \begin{cases}
    2^{i/2}(\frac{-d_{\Lambda}+2}{2})_{i/2}a_1^{-\Lambda_2+1}a_2^{\Lambda_1+1}&(\text{$i$ is even and $0\le i\le d_{\Lambda}+\delta(2\Lambda_2-1)$})\\
    0&(\text{otherwise})
    \end{cases},\\
    f_{4,i}(a_0)&:=
    \begin{cases}
    2^{\frac{i-1}{2}}(\frac{-d_{\Lambda}+2}{2})_{\frac{i-1}{2}}a_1^{-\Lambda_2+1}a_2^{\Lambda_1+1}&(\text{$i$ is odd and $0\le i\le d_{\Lambda}+(1-\delta)(2\Lambda_2-1)$})\\
    0&(\text{otherwise})
    \end{cases}.
    \end{align*}
    Here, $\delta\in\{0,1\}$ is determined by $d_{\Lambda}\equiv\delta\mod 2$.
    \end{enumerate}
    \end{thm}
    
    \begin{proof}
    We give a proof only for $\lambda\in\Xi_{II}$ in view of Lemma \ref{Whittaker-MVWinv}.
    Different from the previous theorem, we solve the differential equations described with the basis $\{v_i^*\mid 0\le i\le d_{\Lambda}\}$~(cf.~Section \ref{RepMaxCpt}). 
    Plugging the Iwasawa decomposition of the non-compact root vectors into the formulas in Lemma \ref{Diffeq-basic}, from the left $N_0$-invariance and the right $K$-equivariance with respect to $\tl^*$, we deduce the differential equations as follows:
    \begin{align}
        &(\partial_1+\Lambda_2-i-2)\phi_i(a_0)+(\partial_2+\Lambda_1-i-2)\phi_{i+2}(a_0)
        =0 \tag{$A_i$},\\
        &(\partial_2-\Lambda_1-i)\phi_i(a_0)+(\partial_1-2\Lambda_1+\Lambda_2+i)\phi_{i+2}(a_0)
        =0\tag{$B_i$},\\
        &j(\partial_2-\Lambda_1+j-1)\phi_{j-1}(a_0)-(\Lambda_1-\Lambda_2-j)(\partial_1-\Lambda_1+j-1)\phi_{j+1}(a_0)
        =0\tag{$C_i$}
    \end{align}
    Here, $i$ runs over $0\le i\le d_{\Lambda}-2$ and $j$ runs over $0\le j\le d_{\Lambda}$.
    
    From $(A_i)$ and $(B_i)$ we obtain
    \begin{equation}
        \left((\partial_2-1)^2-\Lambda_2^2-(\partial_1-\Lambda_1-1)(\partial_1-\Lambda_1+2\Lambda_2-1)\right)\phi_i(a_0)=0 \qquad\text{for any $0\le i\le d_{\Lambda}$} \tag{$D_i$}
    \end{equation}
    and from $i(B_{i-1})-(C_i)$
    \begin{equation}
        (\partial_1-\Lambda_1-1)\phi_i(a_0)=0 \qquad \text{ for any $0\le i\le d_{\Lambda}-1$} \tag{$E_i$}.
    \end{equation}
    By $(D_i)$ and $(E_i)$, for any $1\le i\le d_{\Lambda}$, we see
    \[
    \phi_i(a_0)=C_ia_1^{\Lambda_1+1}a_2^{\Lambda_2+1}+D_ia_1^{\Lambda_1+1}a_2^{-\Lambda_2+1}
    \]
    with constants $C_i,~D_i$ depending only on $i$. 
    We determine $C_i,~D_i$ and $\phi_0(a)$ by $(A_i),~(B_i)$ and $(C_{d_{\Lambda}})$. The differential equations $(A_i)$s imply
    \begin{align*}
       (\partial_1+\Lambda_2-2)\phi_0(a_0)+C_2(\Lambda_1+\Lambda_2-1)a_1^{\Lambda_1+1}a_2^{\Lambda_2+1}+D_2(\Lambda_1-\Lambda_2-1)a_1^{\Lambda_1+1}a_2^{-\Lambda_2+1}&=0,\\
        (\Lambda_1+\Lambda_2-i-1)(C_i+C_{i+2})&=0,\\
        \quad(\Lambda_1+\Lambda_2-i-1)D_i+(\Lambda_1-\Lambda_2-i-1)D_{i+2}&=0,
    \end{align*}
    for any $1\le i\le d_{\Lambda}-2$, and $(B_i)$s imply
    \begin{align*}
       (\partial_2-\Lambda_1)\phi_0(a_0)-(\Lambda_1-\Lambda_2-1)\left(C_2a_1^{\Lambda_1+1}a_2^{\Lambda_2+1}+D_2a_1^{\Lambda_1+1}a_2^{-\Lambda_2+1}\right)
        &=0,\\
        (\Lambda_1-\Lambda_2-i-1)(C_i+C_{i+2})
        &=0,\\\
        (\Lambda_1+\Lambda_2-i-1)D_i+(\Lambda_1-\Lambda_2-i-1)D_{i+2}
        &=0,
    \end{align*}
    for any $1\le i\le d_{\Lambda}-2$.
      We now first assume that $(\partial_2-\Lambda_1)\phi_0(a_0)=0$.
      The fourth equation of the six above and $(A_0)$ lead to
      \[
      \phi_0(a_0)=C_0a_1^{2-\Lambda_2}a_2^{\Lambda_1},~C_2=D_2=0,
      \]
      and the fifth and sixth~(or second and third) thus lead to
      \[
      C_{2i}=D_{2i}=0~(i\ge 1).
      \]
      Next, suppose that $(\partial_2-\Lambda_1)\phi_0(a_0)\not=0$. We then see that we can put $\phi_0(a)=C_0a_1^{\Lambda_1+1}a_2^{\Lambda_2+1}+D_0a_1^{\Lambda_1+1}a_2^{-\Lambda_2+1}$ in view of the fourth equation.
      We can therefore express $\phi_i(a)$ as
      \[
      \phi_i(a_0)=C_ia_1^{\Lambda_1+1}a_2^{\Lambda_2+1}+D_ia_1^{\Lambda_1+1}a_2^{-\Lambda_2+1}
      \]
    for any $1\le i\le d_{\Lambda}$.
    In view of $(C_{d_{\Lambda}})$ and the six differential equations above the determination of $C_i$ and $D_i$ is reduced to the following three equations:
    \begin{align*}
        &D_{d_{\Lambda}-1}=0,\\
        &(\Lambda_1+\Lambda_2-i-1)(C_i+C_{i+2})=0,~(\Lambda_1-\Lambda_2-i-1)(C_i+C_{i+2})=0,\\
        &(\Lambda_1+\Lambda_2-i-1)D_i+(\Lambda_1-\Lambda_2-i-1)D_{i+2}=0,
    \end{align*}
    for any $0 \leq i \leq d_{\Lambda}-2$.
    We verify that $C_i$ is determined recursively from $C_0$ or $C_1$ when $i$ is even or odd, respectively.
    Also, we verify $D_i$ from $D_{d_{\Lambda}}$ or $D_{d_{\Lambda}-1}$. More precisely, we have to be careful about the parity of $d_{\Lambda}$ for the determination of $D_i$s. 
    As a result of such a careful analysis of the recurrence relations on $C_i$ and $D_i$, we obtain the results in the assertion.
    \end{proof}
    
    \begin{rem}
    In Theorem \ref{Deg-Whitt-Siegel}, we have to assume the moderate growth condition for the Whittaker functions to exclude those of non-moderate growth. In fact, we are motivated by studies of automorphic forms. However, we do not have to assume such a condition for Theorem \ref{Deg-Whitt-Borel} since every solution is of moderate growth for this case.
    \end{rem}
    
\subsection{Fourier-Jacobi type spherical functions for the trivial character of $N_J$}
    
    We next examine interested in the Fourier-Jacobi type spherical functions introduced by Hirano \cite{2001_Hirano} for the irreducible unitary representations of $G_J$ of the form 
    \[
    \rho_{\pi_1,0}=\pi_1\boxtimes 1_{N_J}
    \]
    with an irreducible unitary representation $\pi_1$ of $\SL_2(\R)$ and the trivial character $1_{N_J}$ of $N_J$. 
    Recall that, by $\cD^{+}_n$~(resp.~$\cD^{-}_n$), we denote a discrete series representation of $\SL_2(\R)$ with lowest weight $n\in\Z_{>1}$~(resp.~highest weight $-n$).
    Note that the Fourier-Jacobi type spherical functions are determined by their restriction to $A_J^\infty$~(cf.~Section \ref{BN}) by its left $\rho_{\pi_1,0}$-equivalence and right $\tl^*$-equivalence.
    The following is due to Hirano \cite[Theorems 7.10,~7.11]{2001_Hirano}.
    
    \begin{thm}[Hirano]\label{FJ-sphrical}
    \begin{enumerate}
    \item Let $\calD_{\lambda}$ be a large discrete series representation of $G$ with Harish-Chandra parameter $\lambda\in\Xi_{II}$ and minimal $K$-type $\tl$.
    Then, the Fourier-Jacobi type spherical function for $\calD_{\lambda}$ of type $(\rho_{\pi_1,0},\calD_{\lambda};\tl)$ is non-zero if and only if $\pi_1=\cD_{\lambda_1+1}^+$ or $\cD_{-\lambda_2+1}^+$. 
    
    When $\pi_1=\cD^+_{\lambda_1+1}$, it is given by
    \[
    a^{-\lambda_2+2}(w_{\lambda_1+1}\otimes v_{d_{\Lambda}}^*)
    \]
    up to scalars.
    When $\pi_1=\cD^+_{-\lambda_2+1}$, it is given by
    \[
    a^{\lambda_1+2}
    \left(\sum_{0\le i\le[\frac{\lambda_1+\lambda_2}{2}]}(-1)^i\frac{(-\lambda_2+1)\cdots(-\lambda_2+i)}{i!}w_{-\lambda_2+2i+1}\otimes v_{-2\lambda_2+2i+1}^*
    \right)
    \]
    up to scalars.
    \item Let $\calD_{\lambda}$ be a large discrete series representation of $G$ with Harish-Chandra parameter $\lambda\in\Xi_{III}$ and minimal $K$-type $\tl$. Then, the Fourier-Jacobi type spherical function for $\calD_{\lambda}$ of type $(\rho_{\pi_1,0},\calD_{\lambda};\tl)$ is non-zero if and only if $\pi_1=\cD_{-\lambda_2+1}^-$ or $\cD_{\lambda_1+1}^-$. 
    
    When $\pi_1=\cD^-_{-\lambda_2+1}$, it is given by
    \[
    a^{\lambda_1+2}(w_{\lambda_2-1}\otimes v_0^*)
    \]
    up to scalars.
    When $\pi_1=\cD^-_{\lambda_1+1}$, it is given by
    \[
    a^{-\lambda_2+2}\left(\sum_{0\le i\le[\frac{\lambda_1+\lambda_2}{2}]}(-1)^i\frac{(\lambda_1+1)\cdots(\lambda_1+i)}{i!}w_{-\lambda_1-2i-1}\otimes v_{-\lambda_1-\lambda_2-2i}^*\right)
    \]
    up to scalars.
    \end{enumerate}
    Note that we use the notations $w_l$ and $v_k^*$ of the bases for the representation spaces of the discrete series of $SL_2(\R)$ and $\tl^*$ in Sections \ref{DS-SL(2)} and \ref{RepMaxCpt}.
    \end{thm}
    
\subsection{Representations generated by Whittaker functions}\label{RepWhittFct}
\subsection*{(i) The case of Siegel parabolic subgroup}
    
    We discuss the Whittaker functions associated with the constant terms along the Siegel parabolic subgroup $P_S$ based on the result of the case $c_0\not=0,~c_3=0$ in Section \ref{Deg-Whitt}.
    For this discussion, we need Lemma \ref{Whittaker-formula}.
    
    For an integer $\ell$, let $\delta_\ell$ be the character of $\SO(2)$ defined by $\left(
    \begin{smallmatrix}
    \cos\theta & \sin\theta\\
    -\sin\theta & \cos\theta
    \end{smallmatrix}
    \right)
    \mapsto e^{i\ell\theta}$.
    We now state a well-known fact as the lemma below:
    \begin{lem}\label{triv_Wh_arch}
    For $m\in\R$, let $\psi_m$ be the additive character of $N_{00}:=\{
    \left(
    \begin{smallmatrix}
    1 & x\\
    0 & 1
    \end{smallmatrix}
    \right)
    \mid x\in\R\}$ defined by 
    $
    \psi_m\left(
    \left(
    \begin{smallmatrix}
    1 & x\\
    0 & 1
    \end{smallmatrix}
    \right)
    \right)=\exp(2\pi\sqrt{-1}\,mx). 
    $
    For $m\ge 0$~(resp.~$m\le 0$), the $\psi_m$-Whittaker function $w$ on $\SL_2(\R)$ for holomorphic discrete series $\cD_n^+$~(resp.~anti-holomorphic discrete series $\cD_n^-$) with 
    minimal $\SO(2)$-type $\delta_n$
    (resp.~$\delta_{-n}$) is given by 
    \[
    w\left(
    \begin{pmatrix}
    1 & x\\
    0 & 1
    \end{pmatrix}
    \begin{pmatrix}
    \sqrt{y} & 0\\
    0 & \sqrt{y}^{-1}
    \end{pmatrix}\right)=C \cdot y^{n/2}\exp(-2\pi|m|y)\psi_m(x).
    \]
    Here, $C$ is a constant and $(x,y)$ runs over $\R\times\R_{>0}$.
    Otherwise, $w\equiv 0$.
    \end{lem}
    \begin{proof}
    The Whittaker function $w$ is the image of
    \[
    \Hom_{({\mathfrak{sl}}_2(\bbR),\SO(2))}\left(\cD_n^{\pm},C^{\infty}\text{-}{\rm Ind}_{N_{00}}^{\SL_2(\R)}\psi_{m}\right)\rightarrow
    \Hom_{\SO(2)}\left(\delta_n^{\pm},C^{\infty}\text{-}{\rm Ind}_{N_{00}}^{\SL_2(\R)}\psi_{m}\right)
    \]
    defined by the restriction map of $\cD_{n}^{\pm}$ to $\delta_{\pm n}$.
    For $\cD_n^+$~(resp.~$\cD_n^-$), $w$ is characterized by the Cauchy Riemann condition $V^{-}\cdot w=0$~(resp. its complex conjugate $V^+\cdot w=0$) since $w$ corresponds to the lowest weight vector or the highest weight vector of $\cD_n^+$ or $\cD_n^-$, respectively. 
    This equation is written as an elementary differential equation of rank one, which leads to the explicit formula for $w$ in the assertion. 
    \end{proof}
    
    We now put $y_1:=a_1/a_2, y_2:=a_1a_2$.
    First, let $\lambda\in\Xi_{II}$.
    We focus on the first of the two coefficient functions in Theorem \ref{Deg-Whitt-Siegel} part 1, which can be written as
    \[
    \varphi_i(a_0)
    =(a_1a_2)^{\frac{d_{\Lambda}+2}{2}}W_{{\rm sgn}(c_0)\left(i-\frac{d_{\Lambda}}{2}\right),\frac{\Lambda_1+\Lambda_2-1}{2}}\left(4\pi|c_0|\frac{a_1}{a_2}\right)
    =y_2^{\frac{d_{\Lambda}+2}{2}}y_1^{\frac{\Lambda_1+\Lambda_2}{2}}\exp(-2\pi|c_0|y_1)
    \]
    for $i=\Lambda_1$ when $c_0>0$~(resp.~$i=-\Lambda_2$ when $c_0<0$).
    As a function in $y_1$, this is well-known as a Whittaker function for the holomorphic~(resp.~anti-holomorphic) discrete series of $\SL_2(\R)$ with lowest weight $\Lambda_1+\Lambda_2$ for $c_0>0$~(resp.~highest weight $-\Lambda_1-\Lambda_2$ for $c_0<0$).  
    As for the other, $\varphi_i$ with $i\not=\Lambda_1$ they are understood as the shifts of $y_1^{\frac{\Lambda_1+\Lambda_2}{2}}\exp(-2\pi|c_0|y_1)$ by the Maass weight raising or lowering operators~(cf.~\cite[Proposition 2.2.5,~(2.31),~(2.32)]{1997_Bump}).  
    In fact, when $c_0>0$ or $c_0<0$, in view of the formula (\ref{f1}) or (\ref{f2}), the explicit formula for the former coefficient function in Theorem \ref{Deg-Whitt-Siegel} part 1 satisfies
    \begin{align*}
    \left(
    y_1\frac{d}{dy_1}+\sqrt{-1}\left(2\pi c_0\sqrt{-1}\right)y_1+\left(i-\frac{d_{\Lambda}}{2}\right)
    \right)
    W_{i-\frac{d_{\Lambda}}{2},\frac{\Lambda_1+\Lambda_2-1}{2}}(4\pi|c_0|y_1)
    =
    -W_{i+1-\frac{d_{\Lambda}}{2},\frac{\Lambda_1+\Lambda_2-1}{2}}(4\pi|c_0|y_1),
    \end{align*}
    for $c_0>0$ and
    \begin{align*}
    &\left(y_1\frac{d}{dy_1}+\sqrt{-1}\left(2\pi c_0\sqrt{-1}\right)y_1+\left(i-\frac{d_{\Lambda}}{2}\right)\right)
    W_{-(i-\frac{d_{\Lambda}}{2}),\frac{\Lambda_1+\Lambda_2-1}{2}}(4\pi|c_0|y_1)\\
    &\qquad
    =-\left(\left(\frac{\Lambda_1+\Lambda_2-1}{2}\right)^2-\left(i-\frac{d_{\Lambda}+1}{2}\right)^2\right)W_{-(i+1-\frac{d_{\Lambda}}{2}),\frac{\Lambda_1+\Lambda_2-1}{2}}(4\pi|c_0|y_1)
    \end{align*}
    for $c_0<0$, which is the weight raising or lowering starting from $y_1^{\frac{\Lambda_1+\Lambda_2}{2}}\exp(-2\pi|c_0|y_1)$, respectively.
    
    Second, let $\lambda\in\Xi_{III}$.
    We here focus on the first of the two coefficient functions in Theorem \ref{Deg-Whitt-Siegel} part 2 and have 
    \[
    \varphi_i(a_0)=(a_1a_2)^{\frac{d_{\Lambda}+2}{2}}W_{{\rm sgn}(c_0)\left(i-\frac{d_{\Lambda}}{2}\right),\frac{\Lambda_1+\Lambda_2+1}{2}}\left(4\pi|c_0|\frac{a_1}{a_2}\right)
    =y_2^{\frac{d_{\Lambda}+2}{2}}y_1^{-\frac{\Lambda_1+\Lambda_2}{2}}\exp(-2\pi|c_0|y_1)
    \]
    for $i=-\Lambda_2$ when $c_0>0$~(resp.~$i=\Lambda_1$ when $c_0<0$).
    Note that we use (\ref{f3}) and (\ref{f4}) to obtain this formula.
    As a function in $y_1$, this is well-known as a Whittaker function for the holomorphic~(resp.~anti-holomorphic) discrete series of $\SL_2(\R)$ with lowest weight $-(\Lambda_1+\Lambda_2)$ for $c_0>0$~(resp.~highest weight $\Lambda_1+\Lambda_2$ for $c_0<0$), for which we remark that $\Lambda_1+\Lambda_2<0$ for $\lambda\in\Xi_{III}$. As for the other coefficients, they are understood as the shifts of $y_1^{-\frac{\Lambda_1+\Lambda_2}{2}}\exp(-2\pi|c_0|y_1)$ by the Maass weight raising or lowering operators. 
    In fact, when $c_0>0$ or $c_0<0$, in view of the formula (\ref{f1}) or (\ref{f2}), the explicit formula for the former coefficient function in Theorem \ref{Deg-Whitt-Siegel} part 2 satisfies
    \begin{align*}
    \left(y_1\frac{d}{dy_1}+\sqrt{-1}\left(2\pi c_0\sqrt{-1}\right)y_1+\left(i-\frac{d_{\Lambda}}{2}\right)\right)W_{i-\frac{d_{\Lambda}}{2},\frac{\Lambda_1+\Lambda_2+1}{2}}(4\pi|c_0|y_1)
    =-W_{i+1-\frac{d_{\Lambda}}{2},\frac{\Lambda_1+\Lambda_2+1}{2}}(4\pi|c_0|y_1),
    \end{align*}
    for $c_0>0$ and
    \begin{align*}
    &\left(y_1\frac{d}{dy_1}+\sqrt{-1}\left(2\pi c_0\sqrt{-1}\right)y_1+\left(i-\frac{d_{\Lambda}}{2}\right)\right)W_{-(i-\frac{d_{\Lambda}}{2}),\frac{\Lambda_1+\Lambda_2+1}{2}}(4\pi|c_0|y_1)\\
    &\qquad=-\left(\left(\frac{\Lambda_1+\Lambda_2+1}{2}\right)^2-\left(i-\frac{d_{\Lambda}+1}{2}\right)^2\right)W_{-(i+1-\frac{d_{\Lambda}}{2}),\frac{\Lambda_1+\Lambda_2+1}{2}}(4\pi|c_0|y_1)
    \end{align*}
    for $c_0<0$, which is the weight raising or lowering starting from $y_1^{-\frac{\Lambda_1+\Lambda_2}{2}}\exp(-2\pi|c_0|y_1)$, respectively. 
    From the discussion so far, we deduce the following lemma:
    
    \begin{lem}\label{Wh_siegel}
    Let ${\cal W}_S$ be the $({\fraks\frakl}_2(\R),\mathrm{O}(2))$-module generated by $\{\langle w_S,v\rangle\mid v\in V_{\Lambda}\}$ with a Whittaker function $w_S$ as in Theorem \ref{Deg-Whitt-Siegel}, where $\langle*,*\rangle$ denotes the canonical pairing for $V_{\Lambda}^*\times V_{\Lambda}$. 
    The representation ${\cal W}_S$ is then a subspace of $\cD_{\Lambda_1+\Lambda_2} \oplus \cD_{d_{\Lambda}}$ as $({\fraks\frakl}_2(\R),\mathrm{O}(2))$-modules.
    \end{lem}
    \begin{proof}
    We first note that the Lie algebra of $\SL^{\pm}(\R)$ is ${\fraks\frakl}_2(\R)$ and that its complexification ${\fraks\frakl}_2(\C)$ has a basis
    \[
    \left\{
    V^+=\frac{1}{2}
    \begin{pmatrix}
    1 & \sqrt{-1}\\
    \sqrt{-1} & -1
    \end{pmatrix},
    V^-=\frac{1}{2}
    \begin{pmatrix}
    1 & -\sqrt{-1}\\
    -\sqrt{-1} & -1
    \end{pmatrix},
    U=
    \begin{pmatrix}
    0 & -\sqrt{-1}\\
    \sqrt{-1} & 0
    \end{pmatrix}
    \right\}
    \]
    (cf. Section \ref{DS-SL(2)}), which forms an ${\fraks\frakl}_2$-triple. 
    Theorem \ref{Deg-Whitt-Siegel} and the discussion before this lemma shows us that the $\C$-span of $\{\langle w_S,v\rangle\mid v\in V_{\Lambda}\}$ is a one-dimensional space generated by a lowest weight vector or a highest weight vector, which is characterized by the vanishing with respect to the derivation along $V^-$ or $V^+$, respectively, or the $\C$-span of the weight raises of such lowest weight vector by $V^+$~(resp.~the weight lowerings of such highest weight vector by $V^-$). 
    We thereby see that the $({\fraks\frakl}_2(\R),\SO(2))$-module generated by  $\{\langle w_S,v\rangle\mid v\in V_{\Lambda}\}$ is isomorphic to the holomorphic discrete series specified by $\cD_{\Lambda_1+\Lambda_2}^+$ or $\cD_{d_{\Lambda}}^+$ or to the anti-holomorphic discrete series specified by $\cD_{\Lambda_1+\Lambda_2}^-$ or $\cD_{d_{\Lambda}}^-$. 
    
    Moreover, consider the translations of $\{\langle w_S,v\rangle\mid v\in V_{\Lambda}\}$ by the matrix $
    \left(
    \begin{smallmatrix}
    -1 & 0\\
    0 & 1
    \end{smallmatrix}\right)\in \mathrm{O}(2)\setminus \SO(2)$.
    The $({\fraks\frakl}_2(\R), \SO(2))$-module generated by them is then the contragredient of the $({\fraks\frakl}_2(\R), \SO(2))$-module generated by $\{\langle w_S,v\rangle\mid v\in V_{\Lambda}\}$.
    As a result, we see that the $({\fraks\frakl}_2(\R), \mathrm{O}(2))$-module generated by $\{\langle w_S,v\rangle\mid v\in V_{\Lambda}\}$ is isomorphic to $\cD_{\Lambda_1+\Lambda_2}$ or $\cD_{d_{\Lambda}}$, which is a sum of $\cD_{\Lambda_1+\Lambda_2}^{\pm}$ or $\cD_{d_{\Lambda}}^{\pm}$, respectively.
    \end{proof}

\subsection*{(ii)~The case of Jacobi parabolic subgroup}
    
    Let $\frakL_J$ be the Lie algebra of $L_J$.
    Then, $\frakL_J \isom \bbC \oplus \mathfrak{sl}_2(\bbR)$.
    As an immediate consequence of Theorem \ref{FJ-sphrical} we have the following:
    
    \begin{lem}\label{Wh_jacobi}
    Let ${\cal W}_J$ be the $(\frakL_J, \SO(2))$-module generated by $\{\langle w_J,v\rangle_{\cH}\mid v\in V_{\Lambda}\}$ with a Whittaker function $w_J$ as in Theorem \ref{FJ-sphrical}, where $\langle*,*\rangle_{\cH}$ denotes the canonical pairing $\cH\otimes V_{\Lambda}^*\times V_{\Lambda}\rightarrow\cH$ with the representation space $\cH$ of $\cD_{\lambda_1+1}^{\pm}$ or $\cD_{-\lambda_2+1}^{\pm}$.
    The representation ${\cal W}_J$ is then a subspace of $|\cdot|^{-\lambda_2+2} \boxtimes \cD_{\lambda_1+1}^+ \oplus |\cdot|^{\lambda_1+2} \boxtimes \cD_{\--\lambda_2+1}^+$ when $\lambda\in\Xi_{II}$~(resp. $|\cdot|^{-\lambda_2+2} \boxtimes \cD_{\lambda_1+1}^- \oplus |\cdot|^{\lambda_1+2} \boxtimes \cD_{-\lambda_2+1}^-$ when $\lambda\in\Xi_{III}$).
    \end{lem}
    
\subsection{The case of minimal parabolic subgroup}\label{section_min_parab}
    
    Let $\frakL_0$ be the Lie algebra of $L_0$, the Levi subgroup of $P_0$.
    
    \begin{lem}\label{Wh_Borel}
    Let ${\cal W}_0$ be the $\frakL_0$-module generated by a Whittaker function $w_0$ as in Theorem \ref{Deg-Whitt-Borel}.
    The representation ${\cal W}_0$ is then a subspace of 
    \[
    |\cdot|^{-\Lambda_2+2} \boxtimes |\cdot|^{\Lambda_1}
    \oplus
    |\cdot|^{\Lambda_1+1} \boxtimes |\cdot|^{\Lambda_2+1}
    \oplus
    |\cdot|^{\Lambda_1+1} \boxtimes |\cdot|^{-\Lambda_2+1}
    \]
    when $\lambda\in\Xi_{II}$ and
    \[
    |\cdot|^{\Lambda_1+2} \boxtimes |\cdot|^{-\Lambda_2}
    \oplus
    |\cdot|^{-\Lambda_2+1} \boxtimes |\cdot|^{-\Lambda_1+1}
    \oplus
    |\cdot|^{-\Lambda_2+1} \boxtimes |\cdot|^{\Lambda_1+1}
    \]
    when $\lambda\in\Xi_{III}$.
    \end{lem}
    
    We now discuss the embedding of $\calD_\lambda$ into principal series representations.
    
    \begin{cor}\label{emb_p0_II}
    Let $\lambda \in \Xi_{II}$.
    Let $\mu_1$ and $\mu_2$ be unitary characters of $\bbR^\times$.
    \begin{enumerate}
        \item The principal series representation
        \[
        \Ind_{P_0}^G
        \left(\mu_1|\cdot|^{-\lambda_2} \boxtimes \mu_2|\cdot|^{\lambda_1}\right)
        \]
        contains $\calD_\lambda$ if and only if $\mu_1 = \sgn^{\lambda_2}$ and $\mu_2=\sgn^{\lambda_1+1}$.
        \item The principal series representation
        \[
        \Ind_{P_0}^G
        \left(\mu_1|\cdot|^{\lambda_1}\boxtimes \mu_2|\cdot|^{\lambda_2}\right)
        \]
        contains $\calD_\lambda$ if and only if $\mu_1\mu_2 = \sgn^{\lambda_1+\lambda_2+1}$.
        \item The principal series representation
        \[
        \Ind_{P_0}^G
        \left(\mu_1|\cdot|^{\lambda_1}\boxtimes \mu_2|\cdot|^{-\lambda_2}\right)
        \]
        contains $\calD_\lambda$ if and only if $\mu_1\mu_2 = \sgn^{\lambda_1+\lambda_2+1}$.
        \item The representation $\calD_\lambda$ does not occur in the principal series representations
        \[
        \text{
        $\Ind_{P_0}^G\left(\sgn^{\lambda_1-\lambda_2+1}|\cdot|^{-\lambda_1}\boxtimes|\cdot|^{\lambda_2}\right)$ 
        $\quad$ 
        and 
        $\quad$ 
        $\Ind_{P_0}^G \left(\sgn^{\lambda_1-\lambda_2}|\cdot|^{-\lambda_1}\boxtimes \sgn|\cdot|^{\lambda_2}\right)$
        }
        \]
        as subrepresentations.
    \end{enumerate}
    \end{cor}
    \begin{proof}
    By \cite[Theorem 3.5]{1990_Yamashita}, it suffices to show the ``if parts'' of the three assertions other than the last one.
    In fact, by Lemma \ref{Deg-Whitt-Borel}, the space of solutions of the Dirac-Schmid operator is five dimensional.
    Thus, by \cite[Theorem 3.5]{1990_Yamashita}, there are five principal series representations that contain $\calD_\lambda$.
    The ``if parts" follow from Lemma \ref{emb_ps}, Lemma \ref{emb_pj}, and the double induction formula.
    This completes the proof.
    \end{proof}
    
    Similarly, we obtain the following:
    \begin{cor}\label{emb_p0_III}
    Let $\lambda \in \Xi_{III}$.
    Let $\mu_1$ and $\mu_2$ be unitary characters of $\bbR^\times$.
    \begin{enumerate}
        \item The principal series representation
        \[
        \Ind_{P_0}^G \left(\mu_1|\cdot|^{\lambda_1} \boxtimes \mu_2|\cdot|^{-\lambda_2}\right)
        \]
        contains $\calD_\lambda$ if and only if $\mu_1 = \sgn^{\lambda_2}$ and $\mu_2=\sgn^{\lambda_1+1}$.
        \item The principal series representation
        \[
        \Ind_{P_0}^G \left(\mu_1|\cdot|^{-\lambda_2}\boxtimes \mu_2|\cdot|^{-\lambda_1}\right)
        \]
        contains $\calD_\lambda$ if and only if $\mu_1\mu_2 = \sgn^{\lambda_1+\lambda_2+1}$.
        \item The principal series representation
        \[
        \Ind_{P_0}^G \left(\mu_1|\cdot|^{-\lambda_2}\boxtimes \mu_2|\cdot|^{\lambda_1}\right)
        \]
        contains $\calD_\lambda$ if and only if $\mu_1\mu_2 = \sgn^{\lambda_1+\lambda_2+1}$.
        \item The representation $\calD_\lambda$ does not occur in the principal series representations
        \[
        \text{$\Ind_{P_0}^G\left(\sgn^{\lambda_1-\lambda_2+1}|\cdot|^{\lambda_2}\boxtimes|\cdot|^{-\lambda_1}\right)$  $\quad$ and $\quad$ $\Ind_{P_0}^G\left( \sgn^{\lambda_1-\lambda_2}|\cdot|^{\lambda_2}\boxtimes \sgn|\cdot|^{-\lambda_1}\right)$}
        \]
        as subrepresentations.
    \end{enumerate}
    \end{cor}

\section{Automorphic forms}\label{section_main}
    In this section, we first review the basic definitions and properties of automorphic forms.
    Next, we investigate the cuspidal components of automorphic forms that generate large discrete series representations, and prove the main theorem.
    
\subsection{Automorphic forms on $G(\bbA)$}
	
	Set $K_\bbA = \prod_{v<\infty} G(\bbZ_v) \times K$.
	Let $\calZ$ be the center of $\calU(\frakg_\bbC)$, the universal enveloping algebra of $\frakg_\bbC$.
	We review here the general results of the theory of automorphic forms (for details, see \cite{MW}).
	The same discussion for the symplectic groups of general degree is available in \cite[\S 2]{Horinaga_2}.
	Let $P=MN$ be a standard parabolic subgroup of $G$.
	For a scalar valued smooth function $\phi$ on $N(\bbA)M(\bbQ) \backslash G(\bbA)$, we say that $\phi$ is automorphic if it satisfies the following three conditions:
	\begin{itemize}
	\item $\phi$ is right $K_\bbA$-finite.
	\item $\phi$ is $\mathcal{Z}$-finite.
	\item $\phi$ is slowly increasing.
	\end{itemize}
	We denote by $\mathcal{A}(P \backslash G)$ the space of automorphic forms on $N(\bbA)M(\bbQ) \bs G(\bbA)$.
	For simplicity, we write $\calA(P \bs G) = \calA(G)$ when $P=G$.
	The space $\mathcal{A}(P \backslash G)$ is stable under the action of $G(\bbA)$.
	
	For parabolic subgroups $P$ and $Q$ of $G$, we say that $P$ and $Q$ are associate if the split components $A_P$ and $A_Q$ are $G(\bbQ)$-conjugate. 
	We denote by $\{P\}$ the associated class of the parabolic subgroup $P$.
	For a locally integrable function $\varphi$ on $N_P(\bbQ) \bs G(\bbA)$, set
	\[
	\varphi_{P}(g) = \int_{N_P(\bbQ) \bs N_P(\bbA)} \varphi(ng) \, dn,
	\]
	where $P = M_PN_P$ is the Levi decomposition of $P$ and the Haar measure $dn$ is normalized by
	\[
	\int_{N_P(\bbQ) \bs N_P(\bbA)} \, dn = 1.
	\]
	The function $\varphi_P$ is called the constant term of $\varphi$ along $P$.
	If $\varphi$ lies in $\calA(P \bs G)$, the constant term $\varphi_Q$ is an automorphic form on $N_Q(\bbA)M_Q(\bbQ) \bs G(\bbA)$ for a parabolic subgroup $Q \subset P$.
	We call $\varphi$ cuspidal if $\varphi_Q$ is zero for any standard parabolic subgroup $Q$ of $G$ with $Q \subsetneq P$.
	We denote by $\calA_{\cusp}(P \bs G)$ the space of cusp forms in $\calA(P \bs G)$.
	For a character $\xi$ of the split component $A_P(\bbA)$ of $P(\bbA)$, put
	\[
	\calA(P \bs G)_\xi = \{\varphi \in \calA(P \bs G) \mid \text{$\varphi(ag) = a^{\xi + \rho_P} \varphi(g)$ for any $g \in G(\bbA)$ and $a \in A_P$}\}.
	\]
	Here, $\rho_P$ is the character of $A_P$ corresponding to half the sum of roots of $N_P$ relative to $A_P$ and $a^{\xi+\rho_P}=(\xi\cdot\rho_P)(a)$.
	We define $\calA_\cusp(P \bs G)_\xi$ similarly.
	Set
	\[
	\calA(P \bs G)_Z = \bigoplus_\xi \calA(P \bs G)_\xi, \qquad \calA_\cusp(P \bs G)_Z = \bigoplus_\xi \calA_\cusp(P \bs G)_\xi.
	\]
	Here, $\xi$ runs over all the characters of $A_P(\bbA)$.
	Let $\fraka_P$ be the real vector space generated by coroots associated with the root system of $G$ relative to $A_P$.
	Then, by \cite[Lemma I.3.2]{MW}, there exist canonical isomorphisms
	\begin{align*}\label{finite_function}
	\bbC[\fraka_P] \otimes \calA(P \bs G)_Z \cong \calA(P \bs G), \qquad \bbC[\fraka_P] \otimes \calA_\cusp(P \bs G)_Z \cong \calA_\cusp(P \bs G).
	\end{align*}
	
	For a standard Levi subgroup $M$, set
	\[
	M(\bbA)^1 = \bigcap_{\chi \in \Hom_{\mathrm{conti}}(M(\bbA),\bbC^\times)} \mathrm{Ker}(|\chi|).
	\]
	For a function $f$ on $G(\bbA)$ and $g \in G(\bbA)$, let $f_g$ be the function on $M_P(\bbA)^1$ defined by $m \longmapsto m^{-\rho_P}f(mg)$.
	Put
	\[
	\calA(G)_{\{P\}} = \left\{\varphi \in \calA(G) \,\middle|\, \begin{matrix}\text{$\varphi_{Q,ak}$ is orthogonal to all cusp forms on $M_Q(\bbA_Q)^1$}\\ \text{for any $a \in A_Q, k\in K_\bbA$, and $Q \not \in \{P\}$}\end{matrix}\right\}.
	\] 
	By \cite[Lemma I.3.4]{MW}, the space $\calA(G)_{\{G\}}$ is equal to $\calA_\cusp(G)$.
	More precisely, Langlands \cite{langlands} proved the following result:
	\begin{thm}
	With the above notation, we have
	\[
	\calA(G) = \bigoplus_{\{P\}}\calA(G)_{\{P\}},
	\]
	where $\{P\}$ runs through all associated classes of parabolic subgroups.
	\end{thm}
	
	Let $M$ be a standard Levi subgroup of $G$ and $\tau$ an irreducible cuspidal automorphic representation of $M(\bbA)$.
	We then say that a pair $(M, \tau)$ is a cuspidal datum.
	Take $w \in W$, the Weyl group of $G$.
	Suppose that the Levi subgroup of $M^w$ is a Levi subgroup of a standard parabolic subgroup.
	Put $M^w=wMw^{-1}$ and let $P^w=M^wN^w$ be the standard parabolic subgroup with Levi subgroup $M^w$.
	The irreducible cuspidal automorphic representation $\tau^w$ of $M^w(\bbA)$ is defined by $\tau^w(m')=\tau(w^{-1}m'w)$ for $m'\in M^w(\bbA)$.
	Two cuspidal data $(M, \tau)$ and $(M', \tau')$ are called equivalent if there exists $w\in W$ such that $M' = M^w$ and that $\tau'=\tau^w$.
	
	Let $\mathcal{A}(G)_{(M, \tau)}$ be the subspace of automorphic forms in $\mathcal{A}(G)$ with the cuspidal support $(M, \tau)$, (for the definition, see \cite[\S III.2.6]{MW}).
	Then, the following result is well-known (for example, see \cite[Theorem III.2.6]{MW}).
	
	\begin{thm}\label{cusp_supp_decomp}
	The space $\mathcal{A}(G)$ is decomposed as
	\[
	\mathcal{A}(G)=\bigoplus_{(M, \tau)} \mathcal{A}(G)_{(M, \tau)}.
	\]
	Here, $(M, \tau)$ runs through all equivalence classes of cuspidal data.
	\end{thm}
	
	Let $P$ be a standard parabolic subgroup of $G$ with standard Levi subgroup $M$ and $\pi$ an irreducible cuspidal automorphic representation of $M(\bbA)$.
	Put
	\[
	\calA_\cusp(P\bs G)_\pi = \{\varphi \in \calA(P \bs G) \mid \text{$\varphi_k \in \calA_\cusp(M)_\pi$ for any $k \in K_\bbA$}\}.
	\]
	Here, $\calA_\cusp(M)_\pi$ is the $\pi$-isotypic component of $\calA_\cusp(M)$.
	For an automorphic form $\varphi$, there exists a finite collection of cuspidal data $(M,\tau)$ such that
	\[
	\varphi \in \bigoplus_{(M,\tau)} \calA(G)_{(M,\tau)}
	\]
	by Theorem \ref{cusp_supp_decomp}.
	Let $\varphi_P^\cusp$ be the cuspidal part of $\varphi_P$.
	Then, there exists a finite number of irreducible cuspidal automorphic representations $\pi_1,\ldots,\pi_\ell$ of $M_P(\bbA)$ such that
	\[
	\varphi_P^\cusp \in \bigoplus_{j=1}^\ell \bbC[\mathfrak{a}_P] \otimes \calA_\cusp(P \bs G)_{\pi_j}.
	\]
	We say that a set $\cup_M \{\chi_{\pi_1},\ldots,\chi_{\pi_\ell}\}$ is the set of cuspidal exponents of $\varphi$, where $\chi_{\pi_j}$ is the central character of $\pi_j$.
	For a character $\chi$ of the center of $M_P(\bbA)$, we call the restriction of $\chi$ to $A_P^\infty$ the real part of $\chi$.
	
	Let us now introduce the notion of induced representations on $G(\bbA)$.
	For a character $\mu$ of $\GL_n(\bbA)$, we mean that an automorphic character, i.e., $\GL_n(F)$ is contained in the kernel of $\mu$.
	We regard the representation space of $\pi$ as the space of cusp forms that generate $\pi$.
	For a parabolic subgroup $P = M_PN_P$ and an automorphic representation $\pi$ of $M_P(\bbA)$, we define the space $\Ind_{P(\bbA)}^{G(\bbA)} \left(\pi\right)$ by the space of smooth functions $\varphi$ on $N_{P}(\bbA)P(\bbQ) \bs G(\bbA)$ such that
	\begin{itemize}
	\item $\varphi$ is an automorphic form.
	\item For any $k \in K_\bbA$, the automorphic form $\varphi_k$ lies in the representation space of $\pi$.
	Here, $\varphi_k(m) = \varphi(mk)$.
	\end{itemize}
	
\subsection{Eisenstein series}
	
	In this subsection, we review the theory of Eisenstein series.
	Let $P=MN$ be a parabolic subgroup of $G$ and $(\sigma, V)$ an automorphic subrepresentation of $\calA(M)$.
	We identify the group $\fraka_P^*$ as the character group of $A_P^\infty$.
	Here, $\fraka_P^{*}$ is the dual space of $\fraka_P$.
	For $f \in \Ind_{P(\bbA)}^{G(\bbA)}(\sigma)$ and $\lambda \in \fraka_P^*$, set
	\[
	f_\lambda(g) = (m_P(g))^\lambda f(g).
	\]
	Here, $m_P$ is defined by $m_P(nmk)=p$ for $n \in N(\bbA)$, $m \in M(\bbA)$ and $k \in K_\bbA$.
	By $\lambda \in \fraka_P^*$, the function $(m_P(g))^\lambda=\lambda(m_P(g))$ is well-defined.
	We call the function $f_\lambda$ constructed above a standard section.
	We define the Eisenstein series by
	\[
	E(g,\lambda,f)=\sum_{\gamma \in P(\bbQ)\bs G(\bbQ)} f_\lambda(\gamma g).
	\]
	Note that if $\sigma$ is cuspidal and $E(g,\lambda,f)$ is holomorphic at $s=s_0$, one has
	\[
	E(g,s_0,f) \in \calA(G)_{\{P\}}.
	\]
	The following theorem is due to Langlands \cite{langlands}.
	
	\begin{thm}\label{Eis_ser_conv}
	Let $P=MN$ be a $\bbQ$-parabolic subgroup of $G$ and $(\sigma,V)$ an automorphic subrepresentation of $\calA(M)$ such that $\sigma$ is $A_P^\infty$-invariant and any automorphic form in $V$ is square-integrable.
	Then, $E(s,\lambda,f)$ converges absolutely and uniformly on a compact set in $g$ if $\mathrm{Re}(\langle \lambda - \rho_P, \alpha\rangle)>0$ for any root $\alpha$ in $N$ relative to $A_P^\infty$.
	Moreover, $E(g,\lambda,f)$ can be meromorphically continued to $\fraka_P^*$.
	\end{thm}
	
	Next we consider the constant terms of Eisenstein series.
	For a standard section $f_\lambda$ and an element $w$ of the Weyl group, set
	\[
	M_{w,\lambda} f_\lambda(g) = \int_{N\cap wNw^{-1} \bs N}f_\lambda(w^{-1}ng)\, dn
	\]
	if the integral converges.
	The integral is called the intertwining operator.
	This converges absolutely for some half plane and is meromorphically continued to a whole space $\fraka_P^*$.
	The constant terms of Eisenstein series can be described by intertwining operators $M_{w,\lambda}$ as follows:
	
	\begin{thm}\label{const_term_Eis_ser}
	Let $P=MN$ and $P'=M'N'$ be parabolic subgroups of $G$.
	Let $\sigma$ be an automorphic cuspidal representation of $M(\bbA)$ on $\calA_\cusp(A_P^\infty \bs M)$.
	For a standard section $f_\lambda$ of $\Ind_{P(\bbA)}^{G(\bbA)}(\sigma \otimes \lambda)$, let $E(g,\lambda,f)_{P'}$ be the constant term along $P'$.
	Put
	    \[
	    W(M,M')=
	    \left\{w \in W \,\middle|\, \begin{matrix}\text{$w^{-1}(\alpha)>0$ for any root $\alpha$ in $M'$ relative to $L_0$,}\\
	    \text{and $wMw^{-1}$ is a standard Levi subgroup of $M'$}
	    \end{matrix}\right\}.
	    \]
	    We then have
	    \[
	    E(g,\lambda,f)_{P'} = \sum_{w \in W(M,M')} \sum_{\gamma \in M'(\bbQ)\cap P_0(\bbQ) \bs M'(\bbQ)} M_wf_\lambda(\gamma g).
	    \]
	\end{thm}

\subsection{Whittaker functions for holomorphic automorphic forms on $\GL_2$ and $\SL_2$}
	
	We consider the space and Whittaker functions of automorphic forms on $\SL_2(\bbA)$ or $\GL_2(\bbA)$ that generate (holomorphic) discrete series representations at the archimedean place.
	For $H=\SL_2$ or $\GL_2$, let $\calA(A_H^\infty \bs H)_k^{+}$ be the isotypic component of (holomorphic) discrete series representation $\calD_k^+$ of weight $k$ in $\calA(A_H^\infty \bs H)$ and $\calA(\SL_2)^{-}_k$ the isotypic component of anti-holomorphic discrete series representation $\calD_k^{-}$.
	Here, $A_H^\infty$ is the identity component of the center of $H(\bbR)$ in the real topology.
	Let $N$ be the upper unipotent subgroup of $\SL_2$.
	For $h \in \bbZ$ and an automorphic form $\varphi$, put
	\[
	\Wh_h(\varphi)(g) = \int_{N(\bbQ) \bs N(\bbA)} \varphi(ng) \overline{\psi(hn)}\, dn.
	\]
	We then have
	\[
	\varphi (g) = \sum_{h \in \bbZ} \Wh_h(\varphi)(g).
	\]
	
	\begin{lem}\label{triv_Wh}
	Let $\varphi$ be an automorphic form on $\SL_2(\bbA)$.
	If the Whittaker functions $\Wh_h(\varphi)$ are zero for any $h \neq 0$, then $\varphi$ is a constant function.
	\end{lem}
	\begin{proof}
	By the assumption, we have
	\[
	\varphi(ng)=\sum_h \Wh_h(\varphi)(ng)=\Wh_0(\varphi)(g), \qquad n\in N(\bbA), g \in \SL_2(\bbA).
	\]
	Thus, the automorphic form $\varphi$ is left $\SL_2(\bbQ)$ and $N(\bbA)$-invariant.
	Since $\SL_2(\bbA)$ is generated by $\SL_2(\bbQ)$ and $N(\bbA)$, the automorphic form $\varphi$ is left $\SL_2(\bbA)$-invariant. 
	\end{proof}
	
	\begin{lem}\label{Wh_hol_sl}
	For $k \in \bbZ$, let $\varphi$ be an automorphic form of weight $k$ such that $\Wh_h(\varphi)$ generates the holomorphic discrete series representation of weight $k$ as $(\mathfrak{sl}_2(\R),\SO_2(\R))$-modules for any $h >0$ and $\Wh_h(\varphi) = 0$ for any $h<0$.
	Then, $\varphi$ is a holomorphic automorphic form of weight $k$ if $k>2$.
	\end{lem}
	\begin{proof}
	Let $X$ be a weight lowering operator.
	Then, $X\cdot \varphi$ is of weight $k-2$ and satisfies the condition as in Lemma \ref{triv_Wh}.
	In fact, since the Whittaker functions $\Wh_h(\varphi)$ are of weight $k$, the function $\Wh_h(\varphi)$ must be a highest weight vector of $\calD_k^+$.
	Thus, we have $X\cdot\varphi=0$ by $k>2$.
	By definition, $\varphi$ is holomorphic.
	\end{proof}
	
	By the complex conjugate, we obtain the following:
	\begin{lem}\label{Wh_anti_hol_sl}
	For $k \in \bbZ$, let $\varphi$ be an automorphic form such that $\Wh_h(\varphi)$ generates the anti-holomorphic discrete series representation of weight $k$ as $(\mathfrak{sl}_2(\bbR),\SO_2(\bbR))$-modules for any $h < 0$ and $\Wh_h(\varphi) = 0$ for any $h>0$.
	Then, $\varphi$ is an anti-holomorphic automorphic form of weight $k$ if $k < -2$.
	\end{lem}
	\begin{proof}
	Take the complex conjugate in the proof of Lemma \ref{Wh_hol_sl}.
	\end{proof}
	
	We compute the structures of $\calA(\GL_2)_k,  \calA(\SL_2)_k^{+}$, and $\calA(\SL_2)_k^{-}$.
	Let $B$ (resp.~$B_\GL$) be the upper triangle subgroup of $\SL_2$ (resp.~$\GL_2$).
	
	\begin{prop}\label{prop_sl_gl}
	Take a weight $k \in \bbZ_{>0}$.
	\begin{enumerate}
	    \item 
	If $k\geq 3$ and $H=\SL_2$, one has
	\[
	\calA(H)_{k}^\pm
	\isom
	\bigoplus_\pi \pi \oplus \bigoplus_{\mu} \left(\bigotimes_{v < \infty}\Ind_{B(\bbQ_v)}^{\SL_2(\bbQ_v)}\left(\mu_v|\cdot|^{k-1}\right)\right)\otimes\calD^{\pm}_k,
	\]
	where $\mu$ runs over all Hecke characters with $\mu_\infty=\sgn^{k}$ and $\pi$ runs over all irreducible cuspidal automorphic representations of $\SL_2(\bbA)$. 
        \item 
	If $k\geq 3$ and $H=\GL_2$, one has
	\[
	\calA(A_H^\infty \bs H)_{k}
	\isom
	\bigoplus_\pi \pi \oplus \bigoplus_{\mu_1,\mu_2} \left(\bigotimes_{v < \infty}\Ind_{B_{\GL}(\bbQ_v)}^{\GL_2(\bbQ_v)}\left(\mu_{1,v}|\cdot|^{(k-1)/2}\boxtimes\mu_{2,v}|\cdot|^{-(k-1)/2}\right)\right)\otimes\calD_k,
	\]
	where $\pi$ runs over all irreducible cuspidal automorphic representation in $\calA_\cusp(A_H^\infty\bs H)$ and $\mu_i$ runs over all Hecke characters with $\mu_{1,\infty} \mu_{2,\infty}^{-1}=\sgn^{k}$.
	\end{enumerate}
	\end{prop}
	\begin{proof}
	Recall $\calA(A_H^\infty \bs H)_{\{H\}} = \calA_\cusp(A_H^\infty \bs H)$ and $A_{\SL_2}^\infty = 1$.
	Hence, for the first assertion, it suffices to show the isomorphism
	\[
	\calA(H)_{k}^\pm \cap \calA(H)_{\{B\}} \isom  \bigoplus_{\mu} \left(\bigotimes_{v<\infty}\Ind_{B(\bbQ_v)}^{\SL_2(\bbQ_v)}\left(\mu_v|\cdot|^{k-1}\right)\right) \otimes \calD^\pm_{k}.
	\]
	Take an automorphic form $\varphi \in \calA(H)^\pm_{k} \cap \calA(H)_{\{B\}}$.
	By Lemma \ref{triv_Wh_arch}, the constant term $\varphi_B$ along $B$ lies in the space
	\[
	\varphi_B \in \{f \in \calA(B \bs \SL_2) \mid \text{$f(\diag(a,a^{-1})) = a^k$ for any $a \in A_B^\infty$}\}.
	\]
	The right-hand side is canonically isomorphic to 
	\[
	\bigoplus_{\mu} \bigotimes_{v}\Ind_{B(\bbQ_v)}^{\SL_2(\bbQ_v)}\left(\mu_v|\cdot|^{k-1}\right),
	\]
	where $\mu$ runs over all Hecke characters.
	Since $\varphi$ generates $\calD_k^\pm$, the constant term $\varphi_{B}$ is so.
	Thus, we may regard $\varphi_B$ as an element of
	\begin{align*}
		\bigoplus_{\mu} 
		\left\{
		f \in \bigotimes_{v}\Ind_{B(\bbQ_v)}^{\SL_2(\bbQ_v)}\left(\mu_v|\cdot|^{k-1}\right) \,\middle|\, \text{$f$ generates $\calD_k^\pm$}
		\right\}
		=\bigoplus_{\mu, \mu_\infty=\sgn^k} \left(\bigotimes_{v < \infty}\Ind_{B(\bbQ_v)}^{\SL_2(\bbQ_v)}\left(\mu_v|\cdot|^{k-1}\right)\right)\otimes\calD^{\pm}_k.
	\end{align*}
	Conversely, take an element 
	\[
	\text{$f \in \left(\bigotimes_{v<\infty}\Ind_{B(\bbQ_v)}^{\SL_2(\bbQ_v)}\left(\mu_v|\cdot|^{k-1}\right)\right)\otimes\calD_k^{\pm}$ ~with~ $\mu_\infty=\sgn^k$}.
	\]
	Let $f_s$ be the standard section such that $f_{k-1}=f$.
	Then, the Eisenstein series $E(g,s,f)$ converges absolutely and the constant term $E(g,s,f)_B$ equals to $f$ at $s=k-1$.
	Thus, the constant term along $B$ induces the desired isomorphism.
	
	For the second assertion, take an automorphic form $\varphi \in \calA(A_{\GL_2}^\infty \bs \GL_2) \cap \calA(A_{\GL_2}^\infty \bs \GL_2)$.
	We apply Lemma \ref{triv_Wh_arch} to $\varphi|_{\SL_2(\bbR)}$.
	Since $\varphi$ generates the discrete series representation of weight $k$ as a $(\mathfrak{gl}_2(\bbR),\mathrm{O}_2(\bbR))$-module, one has
	\[
	\varphi_B|_{\SL_2(\bbR)} \in  \Ind_{B(\bbR)}^{\SL_2(\bbR)}\sgn^{k}|\cdot|^{k-1}.
	\]
	Hence, we obtain
	\[
	\varphi_B|_{\GL_2(\bbR)} \in \bigoplus_{\mu_1,\mu_2} \Ind_{B_{\GL}(\bbR)}^{\GL_2(\bbR)}\mu_1|\cdot|^{(k-1)/2}\boxtimes\mu_2|\cdot|^{-(k-1)/2}
	\]
	by the $A_H^\infty$-invariance of $\varphi$.
	Here, $\mu_1$ and $\mu_2$ runs over all Hecke characters with $\mu_{1,\infty}\mu_{2,\infty}^{-1}=\sgn^k$.
	We thus have
	\[
	\varphi_B \in \bigoplus_{\mu_1,\mu_2, \mu_{1,\infty}\mu_{2,\infty}^{-1}=\sgn^k} \left(\bigotimes_{v<\infty}\Ind_{B_{\GL}(\bbQ_v)}^{\GL_2(\bbQ_v)}\left(\mu_{1,v}|\cdot|^{(k-1)/2}\boxtimes\mu_{2,v}|\cdot|^{-(k-1)/2}\right)\right)\otimes\calD_k
	\]
	by the same reason discussed above.
	The converse is the same as the $\SL_2$-case.
	This completes the proof.
	\end{proof}

\subsection{Main theorem}
	
	For a parameter $\lambda \in \Xi_{II} \cup \Xi_{III}$, let $\calL_\lambda$ be the space of automorphic forms that generate the large discrete series representation $\calD_\lambda$.
	Put $\calL_{\lambda,\{P\}}=\calL_\lambda \cap \calA(G)_{\{P\}}$ and $\calL_{\lambda, (M,\pi)} = \calL_\lambda\cap \calA(G)_{(M,\pi)}$.
	We study the space $\calL_{\lambda, (M,\pi)}$ in terms of constant terms and induced representations.
	We first consider the case where $P=P_J$.
	
	\begin{thm}\label{II_main_thm_P_J}
	Let $P=P_J$ and $\lambda = (\lambda_1, \lambda_2) \in \Xi_{II}$.
	Take a cuspidal datum $(M,\pi)$ such that $M=L_J$.
	We denote by $\pi = \mu \boxtimes \sigma$ the outer tensor product decomposition associated with $L_J(\bbA) = \GL_1(\bbA) \times \SL_2(\bbA)$.
	\begin{enumerate}
	    \item If $\calL_{\lambda, (M,\pi)} \neq 0$, the archimedean component of $\sigma$ is $\calD_{\lambda_1+1}^+$ or $\calD^+_{-\lambda_2+1}$.
	    Suppose that $\sigma_\infty = \calD^+_{\lambda_1+1}$ (resp.~$\sigma_\infty=\calD^+_{-\lambda_2}$).
	    We then have $\mu_\infty=\sgn^{\lambda_2}$ (resp.~$\mu_\infty=\sgn^{\lambda_1}$).
	    \item If $(\mu_\infty, \sigma_\infty)=(\sgn^{\lambda_2}, \calD_{\lambda_1+1}^+)$ and $-\lambda_2>2$, the constant term along $P$ induces the isomorphism
	    \[
	    \calL_{\lambda,P} \isom \left(\bigotimes_{v <\infty}\Ind_{P_J(\bbQ_v)}^{G(\bbQ_v)}\left(\mu_v|\cdot|^{-\lambda_2}\boxtimes \sigma_v \right)\right)\otimes \calD_\lambda.
	    \]
	    \item If $(\mu_\infty, \sigma_\infty)=(\sgn^{\lambda_1}, \calD_{-\lambda_2+1}^+)$ and $\lambda_1>2$, the constant term along $P$ induces the isomorphism
	    \[
	    \calL_{\lambda,P} \isom \left(\bigotimes_{v < \infty}\Ind_{P_J(\bbQ_v)}^{G(\bbQ_v)}\left(\mu_v|\cdot|^{\lambda_1}\boxtimes\sigma_v\right)\right) \otimes \calD_\lambda.
	    \]
	\end{enumerate}
	\end{thm}
	
	\begin{proof}
	Take an automorphic form $\varphi \in \calL_{\lambda,(M,\pi)}$.
	We examine the constant term $\varphi_P$ and its restriction to $L_J(\bbA)$.
	Set $\Phi=\varphi_P|_{L_J(\bbA)}$, an automorphic form on $L_J(\bbA)$.
	By Theorem \ref{FJ-sphrical}, the representation of $A_P^\infty=\bbR_+^\times$ generated by $\Phi$ under the left-inverse translation is semisimple with at most two dimensional.
	The representations are of the form $a \mapsto a^{-(\lambda_1+2)}$ or $a \mapsto a^{-(-\lambda_2+2)}$.
	Since the parameter $\lambda$ is not singular, the representations are non-equivalent.
	We denote by $\Phi = \Phi_1+\Phi_2$ the sum according to the decomposition, i.e., $\Phi_i$ satisfies
	\[
	\Phi_i(ag) = a^{|\lambda_i|+2}\Phi(g), \qquad a \in A_J^\infty, g \in L_J(\bbA).
	\]
	Next, we consider the left translation of $A_P(\bbA)$ on $\Phi_i$.
	Since $\bbQ^\times \bbR_+^\times\bs \bbA^\times$ is compact, $\Phi_i$ generates a finite-dimensional semisimple representation of it.
	Thus, $\Phi_i$ can be decomposed as
	\begin{align}\label{phi_i}
	\Phi_i(ag) = \sum_\mu |a|^{|\lambda_i|+2}\mu(a)\Phi_{i,\mu}(g), \qquad a \in A_P(\bbA), g \in \SL_2(\bbA),
	\end{align}
	where $\mu$ runs over all Hecke characters of $\bbQ^\times \bbR_+^\times \bs\bbA^\times$.
	Here, we use $L_J = A_P \times \SL_2$.
	By Theorem \ref{FJ-sphrical} and Lemma \ref{Wh_jacobi}, the Whittaker functions of $\Phi$ generate a holomorphic discrete series representation of weight $-\lambda_2+1$ if $i=1$ and of weight $\lambda_1+1$ if $i=2$.
	Such automorphic forms are holomorphic by Lemma \ref{Wh_hol_sl}.
	This completes the proof of (1).
	Thus, $\Phi_{i,\mu}|_{\SL_2(\bbA)}$ lies in the isotypic component $\bigoplus_\pi L^2_\cusp(\SL_2)_\pi$ where $\pi$ runs over all irreducible cuspidal automorphic representations such that the archimedean component $\pi_\infty$ is isomorphic to the holomorphic discrete series representation of the weight above.
	
	The constant term along $P_J$ induces the inclusion
	\begin{align}\label{P_J_emb_1}
	\calL_{\lambda,(M,\pi)} \xhookrightarrow{\quad} \Ind_{P_J(\bbA)}^{G(\bbA)}\left(\mu|\cdot|^{s_0}\boxtimes\sigma\right),
	\end{align}
	where
	\[
	s_0 = 
	\begin{cases}
	-\lambda_2 & \text{if $\sigma_\infty=\calD_{\lambda_1+1}$}\\
	\lambda_1 & \text{if $\sigma_\infty=\calD_{-\lambda_2+1}$}
	\end{cases}
	\]
	by (\ref{phi_i}).
	Consider the archimedean component of the right-hand side of (\ref{P_J_emb_1}).
	By Lemma \ref{emb_pj}, if $\sigma_\infty=\calD_{\lambda_1+1}^+$ and $\mu \neq \sgn^{\lambda_2}$, the archimedean component of the induced representation does not contain the large discrete series representation $\calD_\lambda$.
	Thus, if $\calL_{\lambda,(M,\pi)} \neq 0$, the character $\mu$ satisfies \[
	\mu_\infty=
	\begin{cases}
	\sgn^{\lambda_2} & \text{if $\sigma_\infty=\calD_{\lambda_1+1}$}\\
	\sgn^{\lambda_1} & \text{if $\sigma_\infty=\calD_{\lambda_2+1}$}.
	\end{cases}
	\]
	We thus obtain the constant term along $P_J$ that induces the inclusion
	\begin{align}\label{P_J_emb_2}
    \calL_{\lambda,(M,\pi)} \xhookrightarrow{\quad} \left(\bigotimes_{v<\infty}\Ind_{P_J(\bbQ_v)}^{G(\bbQ_v)}\left(\mu_v|\cdot|^{s_0}\boxtimes\sigma_v\right)\right) \otimes \calD_\lambda.
	\end{align}
	
	We next prove the second assertion.
	The proof of the last statement is similar.
	It suffices to show that the inclusion (\ref{P_J_emb_2}) is surjective if $-\lambda_2>2$.
	Take a function $f$ on the right-hand side of (\ref{P_J_emb_2}).
	Let $f_s$ be the standard section such that $f_{s_0} = f$.
	By $s_0>2$, the Eisenstein series $E(g,s,f)$ converges absolutely at $s=s_0$.
	We compute the constant term along $P_J$.
	By Theorem \ref{const_term_Eis_ser}, one has
	\[
	E(g,s,f)_{P_J} = f_s(g) + M_{w,s}f_s(g).
	\]
	Here,
	\[
	M_{w,s}f_s(g) = \int_{N_J(\bbQ) \bs N_{J}(\bbA)}f_s(w_1ng)\,dn, \quad w_1 = 
	\left(
	\begin{array}{cc|cc}
	0&0&-1&0\\
	0&1&0&0\\
	\hline
	1&0&0&0\\
	0&0&0&1
	\end{array}\right).
	\]
	Then, at $s=s_0$, the intertwining integral converges absolutely and defines the intertwining operator
	\[
	\Ind_{P_J(\bbR)}^{G(\bbR)}\left(\mu_\infty|\cdot|^{s_0}\otimes\pi_\infty\right)\longrightarrow \Ind_{P_J(\bbR)}^{G(\bbR)}\left(\mu_\infty|\cdot|^{-s_0}\otimes\pi_\infty\right).
	\]
	Note that $\lambda_2<0$.
	Since the right-hand side contains the Langlands subrepresentation, the large discrete series representation does not occur as the subrepresentation in it.
	Thus, the archimedean component of $M_{w,s}f_s$ is equal to zero at $s=s_0$.
	Since the global integral converges absolutely at $s=s_0$, we have $M_{w,s}f_s|_{s=s_0}=0$ and $E(g,s_0,f)_{P_J}=f$.
	This shows that the map (\ref{P_J_emb_2}) is surjective.
	This completes the proof.
	\end{proof}
	
	The proofs of the following Theorems \ref{III_main_thm_P_J}, \ref{II_main_thm_P_S}, and \ref{III_main_thm_P_S} are completely the same as the proof of Theorem \ref{II_main_thm_P_J}.
	
	\begin{thm}\label{III_main_thm_P_J}
	Let $P=P_J$ and $\lambda \in \Xi_{III}$.
	Take a cuspidal datum $(M,\pi)$ such that $M=L_J$.
	We denote by $\pi = \mu \boxtimes \sigma$ the outer tensor product decomposition associated with $L_J(\bbA) = \GL_1(\bbA) \times \SL_2(\bbA)$.
	
	\begin{enumerate}
	    \item If $\calL_{\lambda, (M,\pi)} \neq 0$, the archimedean component of $\sigma$ is $\calD_{\lambda_1+1}^-$ or $\calD_{-\lambda_2+1}^-$.
	    Suppose that $\sigma_\infty = \calD^-_{\lambda_1+1}$ (resp.~$\sigma_\infty=\calD_{-\lambda_2}^-$).
	    We then have $\mu_\infty=\sgn^{\lambda_2}$ (resp.~$\mu_\infty=\sgn^{\lambda_1}$).
	    \item If $(\mu_\infty, \sigma_\infty)=(\sgn^{\lambda_2}, \calD_{\lambda_1+1}^-)$ and $-\lambda_2>2$, the constant term along $P$ induces the isomorphism
	    \[
	    \calL_{\lambda,P} \isom \left(\bigotimes_{v<\infty}\Ind_{P_J(\bbQ_v)}^{G(\bbQ_v)} \mu_v|\cdot|^{-\lambda_2}\boxtimes\sigma_v\right) \otimes \calD_\lambda.
	    \]
	    \item If $(\mu_\infty, \sigma_\infty)=(\sgn^{\lambda_1}, \calD_{-\lambda_2+1}^-)$ and $\lambda_1>2$, the constant term along $P$ induces the isomorphism
	    \[
	    \calL_{\lambda,P} \isom \left(\bigotimes_{v<\infty}\Ind_{P_J(\bbQ_v)}^{G(\bbQ_v)} \mu_v|\cdot|^{\lambda_1}\boxtimes\sigma_v\right) \otimes \calD_\lambda.
	    \]
	\end{enumerate}
    \end{thm}
    \begin{proof}
    The proof is the same as the proof of Theorem \ref{II_main_thm_P_J} by Lemma \ref{Wh_jacobi} and Lemma \ref{Wh_anti_hol_sl}.
    \end{proof}
    
	We next consider the case where $P=P_S$.
	\begin{thm}\label{II_main_thm_P_S}
	Let $P=P_S$ and $\lambda = (\lambda_1, \lambda_2) \in \Xi_{II}$.
	Take a cuspidal datum $(M,\pi)$ such that $L_S$.
	Suppose that $\pi$ is invariant under $A_{S}^{\infty}$.
	\begin{enumerate}
	    \item If $\calL_{\lambda, (M,\pi)} \neq 0$, the archimedean component of $\pi$ is $\calD_{\lambda_1+\lambda_2+1}$ or $\calD_{\lambda_1-\lambda_2+1}$.
	    \item If $\pi_\infty = \calD_{\lambda_1+\lambda_2+1}$ and $\lambda_1-\lambda_2>3$, the constant term along $P$ induces the isomorphism
	    \[
	    \calL_{\lambda,(M,\pi)} \isom \left(\bigotimes_{v<\infty}\Ind_{P_S(\bbQ_v)}^{G(\bbQ_v)}\left(\mu_v|\cdot|^{(\lambda_1-\lambda_2)/2}\otimes\pi_v\right)\right) \otimes \calD_\lambda.
	    \]
	    \item If $\pi_\infty = \calD_{\lambda_1-\lambda_2+1}$ and $\lambda_1+\lambda_2>3$, the constant term along $P$ induces the isomorphism
	    \[
	    \calL_{\lambda,(M,\pi)} \isom \left(\bigotimes_{v<\infty}\Ind_{P_S(\bbQ_v)}^{G(\bbQ_v)}\left(\mu_v|\cdot|^{(\lambda_1+\lambda_2)/2}\otimes\pi_v\right)\right) \otimes \calD_\lambda.
	    \]
	\end{enumerate}
	\end{thm}
	\begin{proof}
	The proof is the same as the proof of Theorem \ref{II_main_thm_P_J} by Lemma \ref{Wh_siegel} and Lemma \ref{Wh_hol_sl}.
    \end{proof}
	
	\begin{thm}\label{III_main_thm_P_S}
	Let $P=P_S$ and $\lambda = (\lambda_1, \lambda_2) \in \Xi_{III}$.
	Take a cuspidal datum $(M,\pi)$ such that $M=M_S$.
	Suppose that $\pi$ is invariant under $A_{S}^{\infty}$.
	\begin{enumerate}
	    \item If $\calL_{\lambda, (M,\pi)} \neq 0$, the archimedean component of $\pi$ is $\calD_{-\lambda_1-\lambda_2+1}$ or $\calD_{\lambda_1-\lambda_2+1}$.
	    \item If $\pi_\infty = \calD_{-\lambda_1-\lambda_2+1}$ and $\lambda_1-\lambda_2>3$, the constant term along $P$ induces the isomorphism
	    \[
	    \calL_{\lambda,(M,\pi)} \isom \left(\bigotimes_{v<\infty}\Ind_{P_S(\bbQ_v)}^{G(\bbQ_v)}\left(\mu_v|\cdot|^{(\lambda_1-\lambda_2)/2}\otimes\pi_v\right)\right) \otimes \calD_\lambda.
	    \]
        \item If $\pi_\infty = \calD_{\lambda_1-\lambda_2+1}$ and $-\lambda_1-\lambda_2>3$, the constant term along $P$ induces the isomorphism
	    \[
	    \calL_{\lambda,(M,\pi)} \isom \left(\bigotimes_{v<\infty}\Ind_{P_S(\bbQ_v)}^{G(\bbQ_v)}\left(\mu_v|\cdot|^{-(\lambda_1+\lambda_2)/2}\otimes\pi_v\right)\right) \otimes \calD_\lambda.
	    \]
	\end{enumerate}
	\end{thm}
	\begin{proof}
	The proof is the same as the proof of Theorem \ref{II_main_thm_P_J} by Lemma \ref{Wh_siegel} and Lemma \ref{Wh_anti_hol_sl}.
	\end{proof}
	
	We finally consider the case where $P=P_0$.
	The proof of this case is similar to that of Theorem \ref{II_main_thm_P_J}, but we need some additional discussion.
	\begin{thm}\label{II_main_thm_P_0}
	Let $P=P_0$ and $(M,\pi)$ be a cuspidal data with $M=L_0$.
	Take $\lambda \in \Xi_{II}$.
	We define the Hecke characters $\mu_1$ and $\mu_2$ by
	\[
	\pi(\diag(a_1,a_2,a_1^{-1},a_2^{-1}))=\mu_1(a_1)\mu_2(a_2).
	\]
	\begin{enumerate}
	    \item If $\calL_{P,(M,\pi)} \neq 0$,  $\lambda_1>-\lambda_2+1$ and $\lambda_2 >1$, we have
	    \[
	    \mu_{1,\infty}=\sgn^{\lambda_1+1}, \mu_2=\sgn^{\lambda_2}.
	    \]
	    \item If $\mu_{1,\infty}=\sgn^{\lambda_1+1}, \mu_{2,\infty}=\sgn^{\lambda_2}$, $\lambda_1>-\lambda_2+1$ and $\lambda_2 >1$, the constant term along $P$ induces the isomorphism
	\[
	\calL_{\lambda,(M,\pi)} \isom \left(\bigotimes_{v<\infty}\Ind_{P_0(\bbQ_v)}^{G(\bbQ_v)}\left(\mu_{1,v}|\cdot|^{\lambda_1}\boxtimes\mu_{2,v}|\cdot|^{-\lambda_2}\right)\right)\otimes\calD_\lambda.
	\]
	\end{enumerate}
	\end{thm}
	\begin{proof}
	Let $\varphi \in \calL_{\lambda,\{P_0\}}$.
	We may assume that $\varphi$ generates the minimal $K$-type of $\calD_\lambda$.
	We first consider the constant term of $\varphi$ along $P_0$ and its restriction to $G(\bbR)$.
	By Corollary \ref{emb_p0_II}, the restriction $\varphi|_{G(\bbR)}$ lies in the direct sum
	\begin{align*}
	&\Ind_{P_0(\bbR)}^{G(\bbR)}\left(\sgn^{\lambda_2}|\cdot|^{-\lambda_2}\boxtimes\sgn^{\lambda_1+1}|\cdot|^{\lambda_1}\right)
	\\
	&\qquad\oplus
	\Ind_{P_0(\bbR)}^{G(\bbR)}
	\left(\sgn^{\lambda_1}|\cdot|^{\lambda_1}\boxtimes\sgn^{\lambda_2+1}|\cdot|^{-\lambda_2}\right)
	\oplus
	\Ind_{P_0(\bbR)}^{G(\bbR)}
	\left(\sgn^{\lambda_1+1}|\cdot|^{\lambda_1}\boxtimes\sgn^{\lambda_2}|\cdot|^{-\lambda_2}\right)\\
	&\qquad \qquad \oplus
	\Ind_{P_0(\bbR)}^{G(\bbR)}\left(\sgn^{\lambda_1}|\cdot|^{\lambda_1}\boxtimes\sgn^{\lambda_2+1} |\cdot|^{\lambda_2}\right)
	\oplus
	\Ind_{P_0(\bbR)}^{G(\bbR)}\left(\sgn^{\lambda_1+1}|\cdot|^{\lambda_1} \boxtimes \sgn^{\lambda_2} |\cdot|^{\lambda_2}\right).
	\end{align*}
	We write $\varphi_{P_0}|_{G(\bbR)} = f_0+f_1+f_2+f_3+f_4$ according to the decomposition.
	
	Next, we consider the constant term along $P_S$.
	By Lemma \ref{Wh_siegel}, the restriction of the constant term $\varphi_{P_S}|_{G(\bbR)}$ generates the discrete series representation of weight $\lambda_1+\lambda_2+1$ or $\lambda_1-\lambda_2+1$.
	We thus get
	\[
	\varphi_{P_S}
	\in
	\Ind_{P_S(\bbA)}^{G(\bbA)}
	\left(|\cdot|^{(\lambda_1-\lambda_2)/2}\calA(A_{\GL_2}^\infty \bs \GL_2)_{\lambda_1+\lambda_2+1}
	\oplus
	|\cdot|^{(\lambda_1+\lambda_2)/2}\calA(A_{\GL_2}^\infty \bs \GL_2)_{\lambda_1-\lambda_2+1}\right).
	\]
	By $\varphi \in L_{\lambda,\{P_0\}}$, the restriction $\varphi|_{L_S(\bbA)}$ is orthogonal to all the cusp forms of $L_S(\bbA)$.
	Hence, the restriction $\varphi_{P_S}|_{L_{S}(\bbA)}$ is a sum of Eisenstein series of $\GL_2(\bbA)$ that generate discrete series representations.
	The archimedean component of the constant terms of such Eisenstein series is a summation of the form
	\[
	\Ind_{B_\GL(\bbA)}^{\GL_2(\bbA)}\left(\sgn^{\lambda_1+1+\vep}|\cdot|^{\lambda_1}\boxtimes\sgn^{\lambda_2+\vep}|\cdot|^{\pm\lambda_2}\right)
	\]
	for some $\vep\in\{0,1\}$ by Proposition \ref{prop_sl_gl}.
	Hence, $f_0=0$.
	
	Let us consider the constant terms along $P_J$.
	Similarly, the $\SL_2$-part of the restriction $\varphi_{P_J}|_{L_J(\bbR)}$ generates a discrete series representation.
	By Lemma \ref{emb_pj} (cf.~Lemma \ref{emb_ps}), we obtain $f_2=f_3=f_4=0$.
	Hence, $\varphi_{P_0}=f_1$.
	We conclude that the constant term along $P_0$ induces an inclusion
	\begin{align}\label{incl_borel}
	\calL_{\lambda,\{P_0\}}
	\xhookrightarrow{\quad} \bigoplus_{\omega_1,\omega_2}\Ind_{P_0(\bbA)}^{G(\bbA)}\left(\omega_1|\cdot|^{\lambda_1} \boxtimes\omega_2 |\cdot|^{-\lambda_2}\right),
	\end{align}
	where $\omega_1$ and $\omega_2$ run over Hecke characters with $\omega_{1,\infty}=\sgn^{\lambda_1}$ and $\omega_{2,\infty}=\sgn^{\lambda_2+1}$.
	We thus have
	\[
	\calL_{\lambda,(M,\pi)} \xhookrightarrow{\quad} \Ind_{P_0(\bbA)}^{G(\bbA)}\left(\mu_1|\cdot|^{\lambda_1} \boxtimes\mu_2 |\cdot|^{-\lambda_2}\right).
	\]
	
	Take a function $f$ on the right-hand side of (\ref{incl_borel}) such that $f$ generates $\calD_\lambda$.
	Let $f_{(s_1,s_2)}$ be the standard section such that $f_{(\lambda_1,-\lambda_2)}=f$.
	Then, by the assumption in the statement, the Eisenstein series $E(g,s_1,s_2,f)$ converges absolutely at $(s_1,s_2)=(\lambda_1,-\lambda_2)$.
	To prove the theorem, it suffices to show $E(g,\lambda_1,-\lambda_2,f)_{P_0}=f$.
	By Theorem \ref{const_term_Eis_ser}, one has
	\[
	E(g,s_1,s_2,f)_{P_0}=\sum_{w\in W}M_{w,(s_1,s_2)}f_{(s_1,s_2)}.
	\]
	Note that the intertwining operators $M_{w,(s_1,s_2)}f_{(s_1,s_2)}$ converge absolutely at $(s_1,s_2)=(\lambda_1,-\lambda_2)$.
	Hence, the intertwining operator $M_{w,(s_1,s_2)}$ induces an intertwining map between principal series representations.
	We will show that $M_{w,(s_1,s_2)}f_{(s_1,s_2)}=0$ for $w\neq 1$ at $(s_1,s_2)=(\lambda_1,-\lambda_2)$.
	Suppose $w \neq 1$ and $w(s_1,s_2) \neq (s_1,-s_2)$.
	Then, the corresponding principal series representation does not contain $\calD_\lambda$ by Corollary \ref{emb_p0_II}.
	Thus, $M_{w,(s_1,s_2)}f_{(s_1,s_2)}=0$.
	For $w$ with $w(s_1,s_2) = (s_1,-s_2)$, by \cite[Lemma 7.2(ii)]{2009_Muic} and Lemma \ref{emb_pj}, the function $f$ lies in the kernel of $M_{w,(\lambda_1,-\lambda_2)}$.
	Hence, $M_{w,(\lambda_1,-\lambda_2)}f=0$.
	We thus have
	\[
	E(g,s_1,s_2,f)_{P_0}=f.
	\]
	This completes the proof.
	\end{proof}
	
	For $\lambda \in \Xi_{III}$, we obtain the following by the same proof:
	\begin{thm}\label{III_main_thm_P_0}
	Let $P=P_0$ and $(M,\pi)$ be a cuspidal data with $M=M_{P_0}$.
	Take $\lambda \in \Xi_{III}$.
	We define the Hecke characters $\mu_1$ and $\mu_2$ by
	\[
	\pi(\diag(a_1,a_2,a_1^{-1},a_2^{-1}))=\mu_1(a_1)\mu_2(a_2).
	\]
	\begin{enumerate}
	    \item If $\calL_{P,(M,\pi)} \neq 0$,  $-\lambda_2>\lambda_1+1$ and $\lambda_1 >1$, we have
	    \[
	    \mu_{1,\infty}=\sgn^{-\lambda_2+1}, \mu_2=\sgn^{\lambda_1}.
	    \]
	    \item If $\mu_{1,\infty}=\sgn^{-\lambda_2+1}, \mu_{2,\infty}=\sgn^{\lambda_1}$, $-\lambda_2>\lambda_1+1$ and $\lambda_1 >1$, the constant term along $P$ induces the isomorphism
	\[
	\calL_{\lambda,(M,\pi)} \isom \left(\bigotimes_{v <\infty}\Ind_{P_0(\bbQ_v)}^{G(\bbQ_v)}\left(\mu_{1,v}|\cdot|^{-\lambda_2}\otimes\mu_{2,v}|\cdot|^{\lambda_1}\right)\right)\otimes\calD_\lambda.
	\]
	\end{enumerate}
	\end{thm}
	
	\begin{rem}
	The conditions of the theorems correspond to the absolute convergence of the Eisenstein series.
	Without these conditions, the structure of the spaces $\calL_\lambda$ may be complicated due to the analytic nature of Eisenstein series.
	In the case where $\SL_2$ and the weight is two, we can describe the spaces in terms of nearly holomorphic automorphic forms.
	After we introduce the notation of nearly holomorphy for large discrete series representations, the authors expect that the spaces $\calL_{\lambda,(M,\pi)}$ are explicitly determined when $\calD_\lambda$ is not sufficiently regular.
	\end{rem}
	\begin{rem}\label{diff_NH}
	To conclude the paper, we consider the difference between nearly holomorphic automorphic forms and automorphic forms that generate large discrete series representations.
	The case of nearly holomorphic automorphic forms can be found in \cite{Horinaga_2, Horinaga_3}.
	Let $\calN(G)$ be the space of nearly holomorphic automorphic forms, i.e., $\frakp^-$-finite automorphic forms.
	Put $\calN(G)_{\{P\}}=\calN(G)\cap\calA(G)_{\{P\}}$.
	
	For the case of the Siegel parabolic subgroup, we have
	\[
	\calN(G)_{\{P_S\}} = 0.
	\]
	However, $\calL_{\{P_S\}} \neq 0$.
	
	For the case of the Jacobi parabolic, the spaces $\calN(G)_{\{P_J\}}$ and $\calL_{\lambda, \{P_J\}}$ are non-zero and similar when $\lambda \in \Xi_{II}$.
	In fact, the cuspidal components contributing to the spaces are holomorphic automorphic forms on $\SL_2$.
	One difference is the number of candidates of the weights of the holomorphic automorphic form with non-zero contribution.
	This number is two for the large discrete series representations case and is one for the nearly holomorphic case. 
	When $\lambda \in \Xi_{III}$, the anti-holomorphic automorphic forms contribute to the space $\calL_\lambda$.
	
	For the minimal parabolic case, the situations are also similar.
	Assume $\lambda \in \Xi_{II}$.
	The case $\lambda \in \Xi_{III}$ is similar.
	The spaces can be embedded into only one principal series representation.
	One difference is exponents.
	For the large discrete series representations, the exponent is $(\lambda_1,-\lambda_2)$.
	This lies in the positive chamber.
	However, for the holomorphic case of weight $(\lambda_1,\lambda_2)$, the exponent is $(\lambda_2, \lambda_1)$.
	This does not lie in the positive chamber.
	\end{rem}

\bibliographystyle{alpha}
\bibliography{ref}

\begin{thebibliography}{MgVW87}

\bibitem[Ada14]{2014_Adams}
Jeffrey Adams.
\newblock The real {C}hevalley involution.
\newblock {\em Compos. Math.}, 150(12):2127--2142, 2014.

\bibitem[Bum97]{1997_Bump}
Daniel Bump.
\newblock {\em Automorphic forms and representations}, volume~55 of {\em
  Cambridge Studies in Advanced Mathematics}.
\newblock Cambridge University Press, Cambridge, 1997.

\bibitem[GH11]{2011_Goldfeld_Hundley}
Dorian Goldfeld and Joseph Hundley.
\newblock {\em Automorphic representations and {$L$}-functions for the general
  linear group. {V}olume {II}}, volume 130 of {\em Cambridge Studies in
  Advanced Mathematics}.
\newblock Cambridge University Press, Cambridge, 2011.
\newblock With exercises and a preface by Xander Faber.

\bibitem[HC66]{1966_Harish-Chandra}
Harish-Chandra.
\newblock Discrete series for semisimple {L}ie groups. {II}. {E}xplicit
  determination of the characters.
\newblock {\em Acta Math.}, 116:1--111, 1966.

\bibitem[Hir01]{2001_Hirano}
Miki Hirano.
\newblock Fourier-{J}acobi type spherical functions for discrete series
  representations of {${\rm Sp}(2,\mathbb{R})$}.
\newblock {\em Compositio Math.}, 128(2):177--216, 2001.

\bibitem[Hor21]{Horinaga_1}
Shuji Horinaga.
\newblock Nearly holomorphic automorphic forms on {$\rm SL_2$}.
\newblock {\em J. Number Theory}, 219:247--282, 2021.

\bibitem[Hor22a]{Horinaga_2}
Shuji Horinaga.
\newblock Nearly holomorphic automorphic forms on {${\rm Sp}_{2n}$} with
  sufficiently regular infinitesimal characters and applications.
\newblock {\em Pacific J. Math.}, 316(1):81--129, 2022.

\bibitem[Hor22b]{Horinaga_3}
Shuji Horinaga.
\newblock On the classification of {$(\mathfrak{g}, K)$}-modules generated by
  nearly holomorphic hilbert-siegel modular forms and projection operators.
\newblock {\em arXiv preprint arXiv:2201.06766}, 2022.

\bibitem[Kli90]{90_Klingen}
Helmut Klingen.
\newblock {\em Introductory lectures on {S}iegel modular forms}, volume~20 of
  {\em Cambridge Studies in Advanced Mathematics}.
\newblock Cambridge University Press, Cambridge, 1990.

\bibitem[Kna01]{2001_Knapp}
Anthony~W. Knapp.
\newblock {\em Representation theory of semisimple groups}.
\newblock Princeton Landmarks in Mathematics. Princeton University Press,
  Princeton, NJ, 2001.
\newblock An overview based on examples, Reprint of the 1986 original.

\bibitem[Kos78]{1978_Kostant}
Bertram Kostant.
\newblock On {W}hittaker vectors and representation theory.
\newblock {\em Invent. Math.}, 48(2):101--184, 1978.

\bibitem[Lan06]{langlands}
Robert~P Langlands.
\newblock {\em On the functional equations satisfied by Eisenstein series},
  volume 544.
\newblock Springer, 2006.

\bibitem[MgVW87]{1987_MVW}
Colette M\oe~glin, Marie-France Vign\'{e}ras, and Jean-Loup Waldspurger.
\newblock {\em Correspondances de {H}owe sur un corps {$p$}-adique}, volume
  1291 of {\em Lecture Notes in Mathematics}.
\newblock Springer-Verlag, Berlin, 1987.

\bibitem[Mui09]{2009_Muic}
Goran Muic.
\newblock Intertwining operators and composition series of generalized and
  degenerate principal series for sp(4,r).
\newblock {\em Glasnik Matematicki - GLAS MAT}, 44:349--399, 12 2009.

\bibitem[MW95]{MW}
Colette Moeglin and Jean-Loup Waldspurger.
\newblock {\em Spectral decomposition and Eisenstein series: a paraphrase of
  the scriptures}.
\newblock Number 113. Cambridge University Press, 1995.

\bibitem[Nar21]{2021_Narita}
Hiro-aki Narita.
\newblock Fourier-jacobi expansion of cusp forms on $sp (2,\mathbb{R})$.
\newblock {\em arXiv preprint arXiv:2111.00756}, 2021.

\bibitem[Oda94]{1994_Oda}
Takayuki Oda.
\newblock An explicit integral representation of {W}hittaker functions on
  {${\rm Sp}(2;{\bf R})$} for the large discrete series representations.
\newblock {\em Tohoku Math. J. (2)}, 46(2):261--279, 1994.

\bibitem[Pra19]{2019_Prasad}
Dipendra Prasad.
\newblock Generalizing the {MVW} involution, and the contragredient.
\newblock {\em Trans. Amer. Math. Soc.}, 372(1):615--633, 2019.

\bibitem[Vog78]{1978_Vogan}
David~A. Vogan, Jr.
\newblock Gel'fand-{K}irillov dimension for {H}arish-{C}handra modules.
\newblock {\em Invent. Math.}, 48(1):75--98, 1978.

\bibitem[Wal83]{1983_Wallach}
Nolan~R. Wallach.
\newblock Asymptotic expansions of generalized matrix entries of
  representations of real reductive groups.
\newblock In {\em Lie group representations, {I} ({C}ollege {P}ark, {M}d.,
  1982/1983)}, volume 1024 of {\em Lecture Notes in Math.}, pages 287--369.
  Springer, Berlin, 1983.

\bibitem[Yam90]{1990_Yamashita}
Hiroshi Yamashita.
\newblock Embeddings of discrete series into induced representations of
  semisimple {L}ie groups. {I}. {G}eneral theory and the case of {${\rm
  SU}(2,2)$}.
\newblock {\em Japan. J. Math. (N.S.)}, 16(1):31--95, 1990.

\end{thebibliography}

\end{document}